\documentclass[11pt]{amsart}

\usepackage[usenames,dvipsnames,svgnames,table]{xcolor}
\usepackage[colorlinks=true, pdfstartview=FitV, linkcolor=blue, citecolor=blue, urlcolor=blue]{hyperref}

\usepackage[centering]{geometry}                % See geometry.pdf to learn the layout options. There are lots.
\usepackage{graphicx}
\usepackage{amssymb}
\usepackage{epstopdf} 
\usepackage{lscape} 
\usepackage{xcolor} 
\usepackage[utf8]{inputenc} 
\usepackage{tikz,caption} 
\usepackage{enumitem}
\setlist[itemize]{leftmargin=2em}
\setlist[enumerate]{leftmargin=2em}
\usepackage{booktabs}
\usepackage{multirow}
\usepackage{mathtools}
\usepackage{mathrsfs}
\usepackage[aligntableaux=center, smalltableaux]{ytableau}
\usepackage{youngtab}
\usepackage{verbatim}
\usepackage{faktor}
\usepackage{xspace}
\usepackage{esvect}
\usepackage{stmaryrd}

\definecolor{darkblue}{rgb}{0.0,0,0.7} % darkblue color
\definecolor{darkred}{rgb}{0.7,0,0} % darkred color
\definecolor{darkgreen}{rgb}{0, .6, 0} % darkgreen color

\newcommand{\defncolor}{\color{darkred}}
\newcommand{\defn}[1]{{\defncolor\emph{#1}}} % emphasis of a definition

\newtheorem{theorem}{Theorem}[section]
\newtheorem{proposition}[theorem]{Proposition}

\newtheorem{corollary}[theorem]{Corollary}

\newtheorem{lemma}[theorem]{Lemma}

\theoremstyle{definition}
\newtheorem{definition}[theorem]{Definition}
\newtheorem{example}[theorem]{Example}
\newtheorem{remark}[theorem]{Remark}
\newtheorem{problem}[theorem]{Problem}
\numberwithin{equation}{section}

\def\NN{{\mathbb N}}
\def\CC{{\mathbb C}}
\def\ZZ{{\mathbb Z}}
\def\QQ{{\mathbb Q}}

\newcommand{\FGCCHA}{\textsf{FGCCHA}\xspace}
\newcommand{\FGCCHAs}{\textsf{FGCCHA}s\xspace}
\newcommand{\vecdim}{\overrightarrow{\dim}}
\newcommand{\veccard}{\overrightarrow{\operatorname{card}}}
\newcommand{\OPG}{\operatorname{OPG}}

\newcommand{\NCSym}{\mathsf{NCSym}}

\usepackage[colorinlistoftodos]{todonotes}

\newcommand{\mike}[1]{\todo[size=\tiny,color=green!30]{#1 \\ \hfill --- Mike}}

\title{When are Hopf algebras determined by integer sequences?}
\author{Nicolas Andrews, Lucas Gagnon, F\'{e}lix G\'{e}linas, Eric Schlums, Mike Zabrocki}

\begin{document}
\maketitle
\begin{abstract}
%attempt 1
We study the category of graded Hopf algebras that are
free noncommutative, cocommutative,
graded and connected
from the perspective of the sequences of dimensions of the graded pieces.
We show that a Hopf algebra exists with a given sequence of graded dimensions
if and only if the ``INVERTi'' transformation of the sequence is nonnegative.
We give conditions on the sequences of graded dimensions for two Hopf algebras $H$ and $K$
in this category under which there exists a surjective homomorphism from $H$ to $K$.
We also give conditions such that an isomorphic copy of
$H$ occurs as a Hopf subalgebra of $K$.
\end{abstract}

\begin{comment}
%attempt 2
We study the category of free, graded, connected, cocommutative Hopf algebras
by what can be said from the sequence of graded dimensions.
Let $\vec{a} = (a_1, a_2, a_3, \ldots)$ be a sequence of nonnegative integers,
then a Hopf algebra $H := \bigoplus_{n\geq0} H_n$ such that 
$a_n = \dim H_n$ if and only if a transformation of $\vec{a}$ is nonnegative.
If $H$ and $K$ are graded Hopf algebras then state conditions on the dimension
sequences of $H$ and $K$...
\end{comment}

\setcounter{tocdepth}{1}
\tableofcontents

\section{Introduction}

This paper begins with a question about the Hopf algebra of symmetric functions in noncommutative variables, $\NCSym$~\cite{AgEtAl12, BHRZ05, GebSag01, HNT08, RS06, W36}.  
This is a graded Hopf algebra with dimensions given by the Bell numbers, and many well-known bases of $\NCSym$ are naturally indexed by set partitions.  
The Catalan numbers enumerate several distinguished families of set partitions, including the nonnesting and noncrossing set partitions.  
We recognized that some bases of $\NCSym$ had product and coproduct formulas that were closed when restricted to the nonnesting or noncrossing subsets, meaning that these sets span (possibly distinct) Hopf subalgebras of $\NCSym$ with graded dimensions equal to the Catalan numbers.

\begin{problem}
\label{problem:motivating_calatan_question}
Classify all Hopf subalgebras of $\NCSym$ with graded dimensions equal to the Catalan numbers.
\end{problem}

%Exploring this question led us to consider the more 
%This led us to the natural extension of this question to any interesting sequence 
%   
%\begin{problem}
%\label{problem:Catalan_hopfsub_classify_all}
%Classify all Hopf subalgebras of $\NCSym$ .
%\end{problem}

We resolve Problem~\ref{problem:motivating_calatan_question} up to isomorphism by applying a recent result of Aliniaeifard and Thiem~\cite{AT22} and the Shirshov--Witt Theorem~\cite{S09, Witt56}.  
The former result applies to every member of the full subcategory of free, graded, connected, and cocommutative Hopf algebras, which we will abbreviate as \FGCCHA.  
As stated above, this includes $\NCSym$.  
%This is a consequence of the category in which $\NCSym$ lies, and applies to a much larger class of Hopf algebras.  
%This category is the full subcategory of free, graded, connected, and cocommutative Hopf algebras, which we will denote by \FGCCHA. 
In the context of \FGCCHAs,~\cite[Theorem~12]{AT22} shows that the isomorphism class of an
\FGCCHA $H$ can be fully determined by the \emph{dimension sequence}
\[
\vecdim(H) = \big( \dim(H_{1}), \dim(H_{2}), \dim(H_{3}), \ldots \big)
\qquad\text{where}\qquad
H = \bigoplus_{n \ge 0} H_{n}.
\]
We answer Problem~\ref{problem:motivating_calatan_question} by showing that all Hopf subalgebras of any \FGCCHA (including $\NCSym$) are \FGCCHAs, a fact that we deduce from the Shirshov--Witt Theorem; see Lemma~\ref{lem:ShiWitt}.

Beyond $\NCSym$, much work has gone into understanding specific examples including the \FGCCHAs of noncommutative symmetric functions (see Example~\ref{ex:NSym}), permutations (see Example~\ref{example:hopf_algebras_of_permutations} and \cite{PR02}), and trees~\cite{AS05cc, GL09}.  
At a more general level, properties of the category of \FGCCHAs (or its dual) have been considered in several references~\cite{AT22, Block85, FP24, LodayRonco10, MM65, Reutenauer-FreeLieAlgebras}, and key results appear in Loday's work on triples of operads which includes \FGCCHAs as a sub-category~\cite[Section 4.1]{L08}.  

The particular ingredients of Aliniaefard and Thiem's result comes from Loday's generalization of the  Cartier--Milnor--Moore and Poincar\'{e}--Birkhoff--Witt theorems~\cite[Theorem 4.1.3]{L08}.  These results state that the structure of an \FGCCHA $H$, is completely determined by the graded dimensions of the subspace of primitives in $H$ and the size of a free generating set of $H$, respectively.

Foissy \cite{F12} has studied the category of free and cofree Hopf algebras from a similar perspective. 
In~\cite[Proposition 23]{F12}, Foissy finds that free and cofree Hopf algebras are also determined by the sequences of
graded dimensions.
For a free commutative and co-commutative Hopf algebra, the Cartier--Milnor--Moore theorem
also can be used to characterize the algebra structure by sequences of graded dimensions.
The present work assembles a similar toolbox for the category of \FGCCHAs and its dual.
%by collecting known results from the literature.
%This paper collects some  as a toolbox for what can be deduced about examples in from associated sequences of integers.

Aguiar and Lauve \cite{AL13} also provide
necessary, but not sufficient, conditions on the dimension sequence for the existence of a Hopf subalgebra
of an arbitrary graded and connected Hopf algebra.
Here we provide both necessary and sufficient
conditions under the assumption that we are working in an \FGCCHA.

We collect known results from various references (Theorem~\ref{thm:AT}, Proposition~\ref{prop:LieToFree}, Proposition~\ref{prop:indecomposable generators} Lemma~\ref{lemma:primitive generators}, and Proposition \ref{lem:ShiWitt}), which have not appeared in a single paper before.  This perspective allows us to determine whether maps between \FGCCHAs exist by checking properties of the sequences.  
For example, proving that two \FGCCHAs have equal dimension sequences amounts to showing that they are isomorphic, by Aliniaeifard and Thiem's result.

The sections of this paper are structured as follows.
Sections~\ref{sec:SequencePrelims} and~\ref{sec:HopfAlgebra} state the
relationship between the dimension sequence of an \FGCCHA, the number of free generators,
and the dimension sequence of the Lie algebra of primitives. Section~\ref{sec:SequencePrelims}
establishes preliminary notation about sequences, while Section~\ref{sec:HopfAlgebra}
focuses on the algebraic content.

In Section~\ref{sec:isoclassification} we interpret a result by
Aliniaeifard and Thiem as a characterization of \FGCCHAs by their
dimension sequences and Theorem~\ref{thm:a sequence}
states that a sequence of nonnegative integers is the 
dimension sequence of an \FGCCHA if and only if the INVERTi transform~\cite{OEIS}
of the sequence is nonnegative.
Theorem~\ref{thm:OPGmaps} describes an explicit construction of the
isomorphisms between \FGCCHAs and a canonical representative of each isomorphism class.

In the last two sections we use the association between \FGCCHAs
and sequences from Section~\ref{sec:isoclassification}
to describe the structure of the category of \FGCCHAs.
In Section~\ref{sec:surjection}, Theorem~\ref{thm:surjection} 
shows that there exists a surjective homomorphism from an \FGCCHA with
dimension sequence $\vec{h}$ to an \FGCCHA with dimension sequence 
$\vec{k}$ if and only if the INVERTi transform of $\vec{k}$ is dominated by 
the INVERTi transform of $\vec{h}$.
In the last section, Theorem~\ref{thm:subclassification2} states
that an \FGCCHA  with dimension sequence $\vec{h}$ occurs as a Hopf 
subalgebra of an \FGCCHA with dimension sequence $\vec{k}$ if and only if 
the inverse Euler transform of $\vec{k}$ dominates the inverse Euler transform of $\vec{h}$.  
This can be used to classify all Hopf subalgebras of \FGCCHAs.

\vspace{4ex}

We illustrate our results by describing how they can be applied to resolve Problem~\ref{problem:motivating_calatan_question}.  
As stated above, the dimension sequence of $\NCSym$ is the Bell numbers,
\[
\vecdim(\NCSym) = (1, 2, 5, 15, 52, 203, \ldots)
\qquad\text{see~\cite[\href{https://oeis.org/A000110}{A000110}]{OEIS}.}
\]
Every Hopf subalgebra of $\NCSym$ corresponds to a sequence whose inverse Euler transform is dominated by $(1, 1, 3, 9, 34, 135,\ldots)$~\cite[\href{https://oeis.org/A085686}{A085686}]{OEIS}, which is  the inverse Euler transform of the Bell numbers.  
In particular, the Catalan numbers $(1, 2, 5, 14, 42, 132, \ldots)$~\cite[\href{https://oeis.org/A000108}{A000108}]{OEIS} have inverse Euler transform $(1, 1, 3, 8, 25, 75,\ldots)$~\cite[\href{https://oeis.org/A022553}{A022553}]{OEIS}, which is dominated by~\cite[\href{https://oeis.org/A085686}{A085686}]{OEIS}.  
Therefore $\NCSym$ has a Hopf subalgebra with Catalan dimension sequence, which is unique up to isomorphism; in the following we identify such a Hopf subalgbera and then identify a well-known representative of its isomorphism class.  

%\begin{example}
Our example uses the power-sum basis $\{\mathbf{p}_{\pi} \;|\; \text{$\pi$ is a set partition of $[n]$, $n \ge 0$}\}$ of $\mathsf{NCSym}$; see~\cite{BHRZ05} for details on the Hopf structure on this basis.  
Direct calculation shows that the subset $\{\mathbf{p}_{\pi} \;|\; \text{$\pi$ is noncrossing}\}$ is closed under the product and coproduct, so they span a Hopf subalgebra whose dimension sequence is the Catalan numbers.  
There are many more examples including the one constructed in~\cite{AT20}.

In~\cite{NT05} Novelli and Thibon define a graded connected Hopf algebra $\mathbf{CQSym}$ which has a basis $\mathbf{P}^{\pi}$ indexed by noncrossing set partitions $\pi$.  
This Hopf algebra is manifestly cocommutative and~\cite[Proposition 5.2]{NT05} shows that it is free, so $\mathbf{CQSym}$ is an \FGCCHA.  
Moreover, the number of noncrossing set partitions is the Catalan numbers.

Using Aliniaefard and Thiem's result, $\mathbf{CQSym}$ is isomorphic to every Hopf subalgebra of $\NCSym$ with dimension sequence equal to the Catalan numbers, and in particular to $\CC\operatorname{-span}\{\mathbf{p}_{\pi} \;|\; \text{$\pi$ is noncrossing}\}$.  
It is important to note that this does not immediately produce an isomorphism, as the product and the coproduct on the bases $\mathbf{P}^{\pi}$ from~\cite{NT05} and $\mathbf{p}_{\pi}$ from~\cite{BHRZ05} are not the same.

\section*{Acknowledgement}
The authors are grateful to Nantel Bergeron and Aaron Lauve for providing feedback on an earlier draft of this article.
We are also grateful to the anonymous referees for valuable suggestions and improvements.

\section{Combinatorics of sequences of numbers}
\label{sec:SequencePrelims}

Let $\QQ^{\ZZ_+}$ (respectively $\ZZ^{\ZZ_+}$ and $\NN^{\ZZ_+}$) denote the space of infinite sequences
$(c_{n})_{n \ge 1} = (c_{1}, c_{2}, \ldots)$ of rational (respectively, integer and natural) numbers,
i.e.~$c_{n} \in \QQ$ for all $n \ge 1$.
Given two sequences $\vec{c} = (c_{1}, c_{2}, \ldots)$ and $\vec{d} = (d_{1}, d_{2}, \ldots)$, we write
\[
\vec{c} \le \vec{d} 
\qquad\text{if and only if}\qquad
\text{$c_{n} \le d_{n}$ for all $n \ge 1$}.
\]
Let $\vec{0} \in \QQ^{\ZZ_+}$ denote the zero sequence, so that $\vec{c} \ge \vec{0}$
if and only if $\vec{c}$ consists entirely of nonnegative entries.

Our main results make use of three interrelated sequences that we will denote by $\vec{h}$, $\vec{a}$, and $\vec{p}$.  
In this section we take a purely enumerative perspective to these sequences, assuming only that they satisfy the relation given in Equation~\eqref{eq:gf_relation} and may have values from a field containing $\QQ$.  

However, in later sections these sequences come from an \FGCCHA $H$
and have an interpretation showing that they are nonnegative integers.
For a graded vector space $V = \bigoplus_{n \ge 0} V_{n}$ with $V_{0} = \CC$, let
\[
\vecdim(V) = \left( \dim(V_{n}) \right)_{n \ge 1}.
\]
Similarly, for a graded set $X = \biguplus_{n \ge 1} X^{(n)}$, let
\[
\veccard(X) = \left( |X^{(n)}| \right)_{n \ge 0}~.
\]
This interpretation will provide useful motivation here:
\begin{itemize}
\item $\vec{h} = (h_{1}, h_{2}, \ldots)$ will be the dimension sequence of $H$,

\item $\vec{a} = (a_{1}, a_{2}, \ldots)$ will be the graded numbers of free generators of the algebra, and 

\item $\vec{p} = (p_{1}, p_{2}, \ldots)$ will be the dimension sequence of the Lie algebra of primitives $\mathcal{P}(H)$.

\end{itemize}

To define our sequences, we make use of the fact, recorded in Proposition~\ref{prop:sequences}
below, that any formal power series in $\QQ[\![t]\!]$ with constant term $1$
can be expressed in three equivalent ways,
\begin{equation}
\label{eq:gf_relation}
1 + \sum_{k \geq 1} h_k t^k = \frac{1}{1 - \sum_{m \geq 1} a_m t^m} = \prod_{d \geq 1} \frac{1}{(1-t^d)^{p_d}}~.
\end{equation}
which determines a triple of sequences: $(\vec{h}, \vec{a}, \vec{p})$ with $\vec{h} = (h_{1}, h_{2}, \ldots)$, $\vec{a} = (a_{1}, a_{2}, \ldots)$, and $\vec{p} = (p_{1}, p_{2}, \ldots)$.

\begin{example}
Take $f(t) = 1 + 2t + 3 t^{2} + \cdots \in \QQ[\![t]\!]$, so that $\vec{h} = (2, 3, 4, \ldots)$.
Then $f(t)$ is the power series,
\[
f(t) = \frac{1}{1 - 2 t + t^{2}} = \frac{1}{(1-t)^{2}}
\]
so the remaining sequences are $\vec{a} = (2, -1, 0, \ldots)$ and $\vec{p} = (2, 0, 0, \ldots)$.
\end{example}
A \defn{composition} of $n$ is a sequence $\beta = (\beta_1, \beta_2, \ldots, \beta_\ell)$ of positive integers such that $\beta_{1} + \beta_{2} + \cdots + \beta_{\ell} = n$.  
We refer to the integers $\beta_{i}$ as the parts of $\beta$ and write $\ell(\beta)$ for the length of $\beta$, which is the number of parts.  
We will use the notation $\beta \vDash n$ to indicate that $\beta$ is a composition of $n$.

We say that $\lambda \vDash n$ is a \defn{partition} of $n$ if $\lambda_1 \geq \lambda_2 \geq \cdots \geq \lambda_{\ell(\lambda)}$.  We will indicate that $\lambda$ is a partition
of $n$ with the notation $\lambda \vdash n$.  We will also use the notation $m_d(\lambda)$ to be the number of times
that $d$ appears as a part in $\lambda$.

For a composition $\alpha = (\alpha_{1}, \alpha_{2}, \ldots, \alpha_{\ell}) \vDash n$,
define for the sequence $\vec{c} = (c_{1}, c_{2}, \ldots) \in \QQ^{\ZZ_+}$,
\[
c_{\alpha} = c_{\alpha_{1}} c_{\alpha_{2}} \cdots c_{\alpha_{\ell}}.
\]
Note that $c_{\alpha} = c_{\beta}$ whenever $\beta \vDash n$ is a composition with the same parts as $\alpha$ in a possibly different order, i.e.~ $\beta_{i} = \alpha_{\sigma(i)}$ for some permutation $\sigma$ of the integers $\{1,2, \ldots, \ell(\alpha)\}$.

We will make extensive use of explicit formulas relating the sequences $\vec{h}$, $\vec{p}$, and $\vec{a}$.  The following formula can be derived directly from Equation~\eqref{eq:gf_relation} using standard techniques.  A form of the first part of (iii) below also appears in~\cite[Theorem 2.2]{KK95} as the ``generalized Witt's formula.''

\begin{proposition}
\label{prop:sequences}
Any one sequence $\vec{h}$, $\vec{a}$, or $\vec{p} \in \QQ^{\ZZ_+}$ belongs to a unique triple $(\vec{h}, \vec{a}, \vec{p})$ of sequences that satisfy Equation~\eqref{eq:gf_relation}, given by:
\begin{enumerate}[label = (\roman*), itemsep = 1em]
\item $\displaystyle h_{n}
= \sum_{\beta \vDash n} a_\beta
= \sum_{\lambda \vdash n} \prod_{d \geq 1} \binom{p_d + m_d(\lambda) -1}{m_d(\lambda)}$, 

\item $\displaystyle a_n
= \sum_{\beta \vDash n} (-1)^{\ell(\beta)-1} h_\beta
= \sum_{\lambda \vdash n} (-1)^{\ell(\lambda)-1} \prod_{d \geq 1} \binom{p_d}{m_d(\lambda)}$, and

\item $\displaystyle p_n
= \sum_{d|n} \sum_{\beta \vDash d} \frac{d\cdot \mu(n/d)}{n \cdot \ell(\beta)} a_\beta
= \sum_{d|n} \sum_{\beta \vDash d} \frac{d\cdot \mu(n/d) (-1)^{\ell(\beta)-1}}{n \cdot \ell(\beta)} h_\beta$.

\end{enumerate}
\end{proposition}

\begin{remark} Although not necessarily obvious, the interested reader can verify using an induction argument that if $\vec{a} \in \ZZ^{\ZZ_+}$ or $\vec{h} \in \ZZ^{\ZZ_+}$, then $\vec{p} \in \ZZ^{\ZZ_+}$.
\end{remark}

Proposition \ref{prop:sequences} is biconditional in the sense that given any equality in the proposition, reversing the calculations shows that the corresponding generating function relations in Equation~\eqref{eq:gf_relation}
will hold.  (See e.g.~\cite{AL15} below Equation (33).)

\begin{definition}
\label{defn:sequencetransfer}
The \emph{sequence transfer maps} relate sequences $(\vec{a}, \vec{h}, \vec{p})$ satisfying Equation~\eqref{eq:gf_relation}:
\begin{align*}
\begin{array}{rcl}
\phi_{p, a} : \QQ^{\ZZ_+} & \to & \QQ^{\ZZ_+} \\
\vec{p} & \mapsto & \vec{a}
\end{array},
\qquad
\begin{array}{rcl}
\phi_{h, a} : \QQ^{\ZZ_+} & \to & \QQ^{\ZZ_+} \\
\vec{h} & \mapsto & \vec{a}
\end{array},
\qquad\text{and}\qquad 
\begin{array}{rcl}
\phi_{p, h} : \QQ^{\ZZ_+} & \to & \QQ^{\ZZ_+} \\
\vec{p} & \mapsto & \vec{h}
\end{array}
\end{align*}
as well as their inverses
\[
\phi_{a, p} = \phi_{p, a}^{-1},
\qquad\qquad
\phi_{a, h} = \phi_{h, a}^{-1},
\qquad\qquad\text{and}\qquad\qquad
\phi_{h, p} = \phi_{p, h}^{-1}.
\]
\end{definition}

Following the conventions of \cite{OEIS}, the map $\phi_{p,h}$ is known
as the Euler transform of the sequence
(and consequently $\phi_{h,p}$ is the inverse Euler transform).
Also following the conventions
of \cite{OEIS},
the map $\phi_{a,h}$ is known as the INVERT transform
and $\phi_{h,a}$ is the INVERTi transform.

\begin{example}
\label{ex:catalan}
Take $f(t) = 1 + t + 2 t^{2} + 5t^3 + 14t^4 + \cdots \in \QQ[[t]]$,
so that $h_n$ is equal to the $n^{th}$ Catalan number
$\frac{1}{n+1}\binom{2n}{n}$.  It is well known that $f(t)$ satisfies the functional equation
\[
f(t) = \frac{1}{1 - t - t^2 - 2t^3 - 5t^4 - 14t^5 - \cdots} = \frac{1}{1-t f(t)}
\]
so we have that $\phi_{h,a}(1,2,5,14,42,132,\ldots) = \vec{a} = (1,1,2,5,14,42, \ldots)$
(that is, $a_n$ is the $n-1^{st}$ Catalan number).
We can use Proposition~\ref{prop:sequences} to calculate the first few values of $\vec{p} = (1,1,3,8,25, \ldots)$.
It follows from \cite[\S 5.1.3]{NT05} that $\vec{p}$ is equal to sequence \cite[\href{https://oeis.org/A022553}{A022553}]{OEIS}.
\end{example}

If the sequence $\vec{a}$ consists entirely of nonnegative integers,
then $\vec{h} = \phi_{a, h}(\vec{a})$ and $\vec{p} = \phi_{a, p}(\vec{a})$ will also have combinatorial interpretation in terms of words.
Recall that a word in a set $X$ is a finite sequence $w = w_{1} w_{2} \ldots w_{\ell}$ of ``letters'' $w_{i} \in X$.  
If $X = \biguplus_{n \geq 1} X^{(n)}$ is a graded set, then define the \emph{degree} of a word to be
\[
\mathsf{deg}(w_1 w_2 \ldots w_{\ell}) = \sum_{i = 1}^{\ell} \mathsf{deg}(w_{i})
\qquad\text{where $\mathsf{deg}(x) = n$ for all $x \in X^{(n)}$}.
\]
For a fixed order on $X$, we order the words on $X$ lexicographically.  
The \emph{rotation} of a word $w = w_{1} w_{2} \ldots w_{\ell}$ is the word
\[
\mathsf{cyc}(w) =  w_{2} \ldots w_{\ell} w_{1}.
\]
This defines an operation of order $\ell$ on words of with $\ell$ letters.  
A word is \emph{Lyndon} if it is strictly smaller than each of $\mathsf{cyc}(w), \mathsf{cyc}^{2}(w), \ldots, \mathsf{cyc}^{\ell-1}(w)$.  
For instance, if $X = \{x < y\}$, then $xyxyy$ is a Lyndon word, but neither $xyxy$ nor $xyx$ are: $xyxy = \mathsf{cyc}^{2}(xyxy)$, while $xyx > xxy = \mathsf{cyc}^{2}(xyx)$.

If the sequence $\vec{a}$ is of nonnegative integers rather than
any values from a field, there is a combinatorial interpretation for
$\phi_{a, h}(\vec{a})$ and $\phi_{a, p}(\vec{a})$.
Similar results are well-known and appear in expositions such as~\cite[Chapter 5]{Reutenauer-FreeLieAlgebras}.  
The exact statement below appears in~\cite[p.~566]{AL15}, which the authors attribute in part to~\cite{KK95}.

\begin{proposition}
\label{prop:combinatorialinterpretation}
Let $\vec{a} \in \NN^{\ZZ_+}$ so that there exists a graded set $X = \biguplus_{n \geq 1} X^{(n)}$ with $\veccard(X) = \vec{a}$.
If $\vec{h} = \phi_{a, h}(\vec{a})$ and $\vec{p} = \phi_{a, p}(\vec{a})$, then for each $n \geq 1$, we have:
\begin{enumerate}[itemsep = 0.5em]
\item $h_n$ is equal to the number of words of degree $n$ in the alphabet $X$ for all $n \ge 1$, and 

\item $p_n$ is equal to the number of Lyndon words of degree $n$ in the alphabet $X$ for all $n \ge 1$.
\end{enumerate}
\end{proposition}

\begin{example}\label{ex:balancedLyndon}
For specific examples of sequences,
$\vec{p}$ may have better-known interpretations
than the one given in  Proposition~\ref{prop:combinatorialinterpretation}.
Continuing with Example \ref{ex:catalan}, we have that if
$h_n = a_{n+1} = \frac{1}{n+1}\binom{2n}{n}$, Proposition
\ref{prop:combinatorialinterpretation} says that $p_n$ is equal to the number
of Lyndon words in an alphabet $X = \biguplus_{n\geq1} X^{(n)}$ with $a_n = |X^{(n)}|$.
However,~\cite[\href{https://oeis.org/A022553}{A022553}]{OEIS} states that $p_n$ is also equal to the number of
length $2n$ Lyndon words in the alphabet $\{0 < 1\}$ with an equal number of $0$s and $1$s.
\end{example}

\begin{corollary}
If $\vec{a} \in \NN^{\ZZ_+}$, then $\vec{h} \geq \vec{p} \geq \vec{a} \geq \vec{0}$.
\end{corollary}
\begin{proof}
Let $X = \biguplus_{n \ge 1} X^{(n)}$ be a graded set with $|X^{(n)}| = a_{n}$.  
Considered as a word with one letter, each element $x \in X$ is a Lyndon word, so by Proposition~\ref{prop:combinatorialinterpretation} the inequalities above correspond to the set inclusions
\[
\{\text{degree $n$ words in $X$}\} 
\supseteq \{\text{degree $n$ Lyndon words in $X$}\} 
\supseteq X^{(n)}. \qedhere
\]
%We have that for each $i \geq 1$ and each $x \in X^{(i)}$ is a Lyndon word, hence $a_i \leq p_i$.
%Moreover each Lyndon word of degree $i$ is itself a word of degree $i$, so $h_i \geq p_i$.
\end{proof}

\section{Sequences and Tensor algebras}
\label{sec:HopfAlgebra}

We now give the algebraic context for the results of Section~\ref{sec:SequencePrelims}.  
As in the introduction, we use the abbreviation \FGCCHA to mean \textsf{F}ree \textsf{G}raded \textsf{C}onnected \textsf{C}ocommutative \textsf{H}opf \textsf{A}lgebra, which we now formally define.

We first recall the tensor algebra of a graded vector space $V = \bigoplus_{n \ge 1} V_{i}$:
\begin{equation}
\label{eq:tensoralg}
\mathsf{T}(V) = \bigoplus_{n \ge 0} \Big( \bigoplus_{k \ge 0} V_{k} \Big)^{\otimes n},
\end{equation}
with a graded multiplication given by $\otimes$, where elements of $V_{k}$ have degree $k$.  

\begin{definition}
\label{def:FGCCHA}
An \FGCCHA is a graded cocommutative Hopf algebra $H$ which is freely generated by some graded subset $X = \biguplus_{n \ge 1} X^{(n)}$ of $H$, so that \textit{as algebras}
\[
H \cong \mathsf{T}(\CC X) 
\qquad\text{where}\qquad
\CC X = \bigoplus_{k \ge 0} \CC\operatorname{-span}\{x \in X^{(k)}\}.
\]
Under this isomorphism, the coproduct of $H$ corresponds to some graded algebra homomorphism $\tilde{\Delta}: \mathsf{T}(\CC X) \to \mathsf{T}(\CC X) \otimes \mathsf{T}(\CC X)$.
\end{definition}

We will also consider Lie subalgebras of $\mathsf{T}(V)$, and more generally any \FGCCHA, under the commutator bracket $[x, y] = xy - yx$.

\begin{definition}[{see~\cite[\S 0.2]{Reutenauer-FreeLieAlgebras}}]
\label{defn:freelie}
The free Lie algebra $\mathfrak{L}(X)$ on a graded set $X = \biguplus_{n\ge 1} X^{(n)}$ is the smallest graded subspace of $\mathsf{T}(\CC X)$ that contains $X$ and is closed under the commutator bracket.  
This is a Lie algebra under the commutator bracket, but not a Hopf algebra.
\end{definition}

Given any \FGCCHA $H$, the primitive elements of $H$ are elements of the graded subspace
\[
\mathcal{P}(H) = \{x \in H \;|\; \Delta(x) = x \otimes 1 + 1 \otimes x\}.
\]
To be clear, the generating set $X$ may not be contained in $\mathcal{P}(H)$.  
The bracket operation $[x,y] = xy - yx$ makes $\mathcal{P}(H)$ into a Lie algebra.  
%However it is not the case that $\mathcal{P}(H)$ maps to $\mathfrak{L}(X)$ under the isomorphism in Definition~\ref{def:FGCCHA}, as the generating set $X$ need not be primitive in $H$.  \lucas{TODO: is this clear?}

Given a graded Lie algebra $L$, the (graded) derived subalgebra of $L$ is 
\begin{equation}
\label{eq:derivedLiesubalgebra}
[L, L] = \bigoplus_{n \ge 1} [L, L]_{n}
\qquad\text{where}\qquad
[L, L]_{n} = \CC\operatorname{-span}\{\text{degree-$n$ commutators of $L$}\}.
\end{equation}

The main result of this section relates the dimension sequence of an \FGCCHA to its generating set and primitives.
This result follows from a standard analysis which appears in the literature without being stated
as a numbered result (see for example \cite{AL15} Section 4.2).

\begin{proposition}
\label{prop:SequencesAsHopfStructures}
Let $H$ be a \FGCCHA with generating set $X$.  The triple of sequences
\[
\left(\vecdim(H), \,
\veccard(X),\,
\vecdim\left(\mathcal{P}(H)\right) \right)
\]
satisfies Equation~\eqref{eq:gf_relation}.
\end{proposition}

For the sake of completeness we prove Proposition~\ref{prop:SequencesAsHopfStructures} at the end of the section.  
We will use the Cartier--Milnor--Moore and Poincare--Birkhoff--Witt theorems, stated below.  
First, however,  we give an example to illustrate the result.  

\begin{example}[Noncommutative symmetric functions]
\label{ex:NSym}
Following~\cite{GKLLRT}, let $\mathsf{NSym}$ be the cocommutative Hopf algebra freely generated by $X = \{ \mathbf{h}_1, \mathbf{h}_2, \mathbf{h}_3, \ldots \}$, with grading given by $\mathsf{deg}(\mathbf{h}_{i}) = i$ for all $i \ge 1$ so that $\veccard(X) = (1, 1, 1, \ldots)$.  
%The Hopf algebra $\mathsf{NSym}$ is the graded dual of the Hopf algebra to the quasisymmetric functions $\mathsf{QSym}$.
The bases of the degree $n$ homogeneous components of $\mathsf{NSym}$ are indexed by compositions
of $n$ and hence $\vecdim(\mathsf{NSym}) = (1,2,4,8,\ldots)$.
The dimension sequence of the space of primitives is 
$\vecdim(\mathcal{P}(\mathsf{NSym})) = (1, 1, 2, 3, 6, 9, 18, 30, 56, \ldots)$
whose $n^{th}$ term is equal to the number
of compositions of $n$ that are Lyndon \cite{H07}.
\end{example}

%We now begin to assemble our proof of Proposition~\ref{prop:SequencesAsHopfStructures}.
%As stated above, this is a folklore result which we prove using well-known theorems for the sake of completeness.

Given a graded Lie algebra $L = \bigoplus_{n \ge 1} L^{(n)}$, let the \emph{universal enveloping algebra} of $L$ be
\[
\mathcal{U}(L) = \mathsf{T}(L) \big/ \big\langle x\otimes y - y \otimes x - [x, y] \;|\; \text{$x, y \in L$} \big\rangle.
\]
This is a cocommutative Hopf algebra where the coproduct on $\mathcal{U}(L)$ is defined so that every element of $L$ is primitive.  By \cite[Theorem 1.4]{Reutenauer-FreeLieAlgebras} we have $\mathcal{P}(\mathcal{U}(L))=L$.  

The relationship between $\mathcal{U}(L)$ and $\mathcal{P}$ is further determined by the Cartier--Milnor--Moore Theorem and the Poincar\'{e}--Birkhoff--Witt Theorem.  We state each as they appear in~\cite[Theorem 4.1.3]{L08}, but we refer the reader to~\cite[Theorem 5.18]{MM65} and~\cite{Cartier57} (see also~\cite[Section 3.8]{Cartier2007}) for original statements.

Given any ordered homogeneous basis $\{z_i\}$ of $L$, the Poincar\'{e}--Birkhoff--Witt Theorem
states that the enveloping algebra $\mathcal{U}(L)$ has a homogeneous basis
\[
\{ z_{i_1}z_{i_2}\cdots z_{i_k} \;|\; i_1 \leq i_2 \leq \cdots \leq i_k\}.
\]
This statement carries important information about the dimension sequence of $\mathcal{U}(L)$, namely that 
\[
\vecdim(\mathcal{U}(L))_{n}= \#\{\text{multisubsets $S$ of $\{z_{i}\}$} \;|\; \sum_{z \in S} m_{z}(S) \mathsf{deg}(z) = n\},
\]
where $m_{z}(S)$ denotes the multiplicity of $z$ in $S$.

For any \FGCCHA $H$, the space $\mathcal{P}(H)$ is a Lie algebra under the bracket $[x, y] = xy - yx$.  
The Cartier--Milnor--Moore theorem states there is a Hopf algebra isomorphism
\[
\begin{array}{rcl}
\mathcal{U}(\mathcal{P}(H)) &\to& H \\
\mathcal{P}(H) \ni x & \mapsto&  x.
\end{array}
\]

\begin{proof}[Proof of Proposition~\ref{prop:SequencesAsHopfStructures}]
Let $\vec{h} = \vecdim(H)$, $\vec{a} = \veccard(X)$, and $\vec{p} = \vecdim(\mathcal{P}(H))$.  We will show that $\vec{h} = \phi_{a, h}(\vec{a})$ and $\vec{p} = \phi_{h, p}(\vec{h})$, from which the claim follows.  

For $n \ge 1$, $h_{n}$ is the number of degree $n$ monomials in $X$, which is also the number of degree $n$ words in $X$.  
Applying Proposition~\ref{prop:combinatorialinterpretation}, $\vec{h} = \phi_{a, h}(\vec{a})$.

On the other hand, the Milnor--Moore theorem states that $H$ is
isomorphic to the universal enveloping algebra of $\mathcal{P}(H)$.
By the Poincare--Birkhoff--Witt theorem, for any fixed homogeneous basis
$Y = \biguplus_{n \ge 0} Y^{(n)}$ of $\mathcal{P}(H)$, $h_{n}$ counts
the multisets of $Y$ whose elements have degree summing (with repetition)
to $n$.  Since $|Y^{(n)}| = p_{n}$ for all $n \ge 1$, the number of these
multisets is exactly
\[
\sum_{\lambda \vdash n} \prod_{d \geq 1} \binom{p_d + m_d(\lambda) -1}{m_d(\lambda)},
\]
so by Proposition~\ref{prop:sequences} we have $\vec{p} = \phi_{h, p}(\vec{h})$.
\end{proof}

\section{Characterization of \FGCCHAs by dimension sequence}
\label{sec:isoclassification}

This section gives a classification of isomorphism types of \FGCCHAs via the associated integer sequences.  
This begins with Aliniaeifard and Thiem's result.

\begin{theorem}[{\cite[Theorem 12]{AT22}}]
\label{thm:AT}
Let $H$ and $K$ be \FGCCHAs.  Then $H \cong K$ if and only if $\vecdim(H) = \vecdim(K)$.
\end{theorem}

The result is deduced from standard properties found in~\cite[Section 4.1]{L08}, and similar results appear elsewhere in the literature, see e.g.~\cite[Prop.~1.4]{AS05cc}.  We demonstrate the utility of this result in Example~\ref{example:hopf_algebras_of_permutations}.

%This result may not be surprising to experts in the field.  
%However, we have not found it recorded elsewhere in the literature, and by stating it the authors of~\cite{AT22} generalize a recurrent approach to comparing Hopf algebras; see Example~\ref{example:hopf_algebras_of_permutations}.
%% by doing so~\cite{AT22} 
%\lucas{Consider rewording}{\color{red} With discussions with people who work in Hopf algebra we've come to the conclusion that this result may have been known implicitly prior to Aliniaeifard and Thiem's publication.  However, given the way results are stated in the literature, it does not seem to be used explicitly; see Example~\ref{example:hopf_algebras_of_permutations}.}

We consider here the question of determining which sequences $\vec{h} \in \NN^{\ZZ_+}$
appear as the dimension sequence of an \FGCCHA.  In fact,
we go further by identifying an explicit representative of each isomorphism class of \FGCCHA.  
Recall the free graded Lie algebra from Definition~\ref{defn:freelie} and for each sequence $\vec{a} \in \NN^{\ZZ_{+}}$, let
\begin{equation}
\label{eq:La}
\mathfrak{L}(\vec{a}) = \mathfrak{L}(X_{\vec{a}}) \qquad\text{where}\qquad X_{\vec{a}} = \biguplus_{n \ge 1} \{x^{(n)}_{i} \;|\; 1 \le i \le a_{n}\}.
\end{equation}
Further recall the map $\phi_{h, a}$ from
Definition~\ref{defn:sequencetransfer} and the universal enveloping
algebra map $\mathcal{U}(-)$ defined in Section~\ref{sec:HopfAlgebra}.

\begin{theorem}
\label{thm:a sequence}
There is a bijection
\[
\begin{array}{rcl}
%\mathcal{U}(\mathfrak{L}(-)):
\Phi: 
\left\{ \text{sequences }\vec{a} \in \NN^{\ZZ_+} \right\} & \to & \left\{ \begin{array}{c} \text{Isomorphism classes} \\ \text{of \FGCCHA} \end{array} \right\} \\[2ex]
\vec{a} & \mapsto & [\mathcal{U}(\mathfrak{L}(\vec{a})))]\\
\phi_{h,a}(\vecdim(H)) & \mapsfrom & [H]
\end{array}.
\]
where $[-]$ denotes an isomorphism class.  
In particular, $\vec{h} \in \NN^{\ZZ_+}$ is the dimension sequence of an \FGCCHA $H$ if and only if $\phi_{h, a}(\vec{h}) \in \NN^{\ZZ_+}$.  
Similarly, $\vec{p} \in \NN^{\ZZ_{+}}$ is the dimension sequence for $\mathcal{P}(H)$ of an \FGCCHA $H$ if and only if $\phi_{p, a}(\vec{p}) \in \NN^{\ZZ_{+}}$.
\end{theorem}

We prove the theorem in Section~\ref{sec:classificationproof} following additional examples and results. 

\begin{example}[A cocommutative Hopf algebra on permutations]
\label{example:hopf_algebras_of_permutations}
    Let $\vec{h} = (n!)_{n \geq 1}$, so that $\vec{a} = \phi_{h,a}(\vec{h}) = (1,1,3,13,71,\ldots)$~\cite[\href{https://oeis.org/A003319}{A003319}]{OEIS}.  
    Then $\vec{a}$ counts graded cardinality of the set $\dot{\mathfrak{S}}$ of \emph{connected} permutations, i.e.~those which are not the shifted concatenation of two others~\cite{AS05}.  
    (In the literature these are sometimes called ``indecomposable'' or ``irreducible'' permutations).
    By Theorem~\ref{thm:a sequence} there exists an \textsf{FGCCHA}
    $\mathcal{U}(\mathfrak{L}(\dot{\mathfrak{S}}))$ of $n$th dimension sequence $n!$.  
%    Viewing each permutation as a word in the alphabet $\ZZ_{+}$, the product is shifted
%    concatenation of permutations, and the coproduct is defined such that
%    every indecomposable permutation is primitive.
    
    There are numerous \FGCCHAs with dimension sequence $n!$ in the literature that have superficially different presentations.  
    This includes the Hopf algebras of heap ordered trees and permutation in~\cite{GL09}, the dual of the associated graded of the Malvenuto--Reutenauer Hopf algebra of permutations~\cite{AS05cc}, the Hopf algebra of permutations $\mathbf{\mathfrak{S}Sym}$ defined in~\cite[\S 3]{HNT08}, and the Hopf algebra of permutations ($\mathbb{K}\mathfrak{S}$) from~\cite[\S 5]{Li15}.
    The papers~\cite{AS05cc, GL09, HNT08, Li15} construct explicit isomorphisms between these Hopf algebras, and indeed by Theorem~\ref{thm:AT}, they are all isomorphic to $\mathcal{U}(\mathfrak{L}(\dot{\mathfrak{S}}))$.
    \end{example}

\begin{example}
Many combinatorially interesting sequences correspond to \FGCCHAs, but it is not obvious when this is the case; we illustrate this by considering two closely-related sequences.
\begin{enumerate}
\item $(1,1,2,3,5,8,13,\ldots)$, the Fibonacci sequence

\item $(2,1,3,4,7,11,18,\ldots)$, the Lucas sequence
\end{enumerate}
Let us consider the Fibonacci sequence $\vec{f} = (1,1,2,3,5,8,13,\ldots)$
as if it could be the dimension sequence of some Hopf algebra.
We compute first that $\phi_{h,a}(\vec{f}) = (1,0,1,0,1,0,\ldots)$
and by Theorem \ref{thm:a sequence} conclude
that there is a Hopf algebra $H$ with a basis indexed by
compositions with only odd parts and with one generator at each odd degree.
This Hopf algebra $H$ has dimension sequence equal to $\vecdim(H) = \vec{f}$.
This Hopf algebra is sometimes referred to as the peak algebra~\cite[\S 2]{Bergeron_2002}.

Next consider the Lucas sequence $\vec{\ell} = (2,1,3,4,7,11,18,\ldots)$.
This sequence cannot be the dimension sequence of an \FGCCHA because of the descent in the first position.
However, we can instead consider $\vec{\ell}' = (1,1,3,4,7,11,18,\ldots)$
as if it were the dimension sequence of an \FGCCHA.
We compute $\phi_{h,a}(\vec{\ell}')$ and determine
that the sequence contains negative entries and begins $(1, 0, 2, -1, 2, -3, 3,\ldots)$ and thus we conclude
that there does not exist a Hopf algebra with dimension sequence equal to $\vec{\ell}'$.
\end{example}

The rest of the section concerns the technical question of explicitly realizing the Aliniaeifard--Thiem isomorphism in every possible way; this has further applications in later sections.  
We now know that every \FGCCHA is isomorphic to $\mathcal{U}(\mathfrak{L}(\vec{a}))$ for some $\vec{a} \in \NN^{\ZZ_{+}}$, so it is natural to ask how many isomorphisms there are between these two objects.  
The following definition is motivated by the fact, proved in Theorem~\ref{thm:OPGmaps} below, that the isomorphisms between $\mathcal{U}(\mathfrak{L}(\vec{a}))$ and $H$ are in bijection with certain sequences of generators.

\begin{definition}
\label{def:OPG}
%\lucas{TODO: change the X and Y's in OPG to be a different letter}
%\lucas{Change notation to reflect symbols for OPG: $\vec{A}$ and $\vec{B}$}
For an \FGCCHA $H$, an \emph{ordered primitive generating set} is a sequence of tuples
\[
\vec{A} = \Big(\vec{A}^{(n)} = (\alpha_{1}^{(n)}, \alpha_{2}^{(n)}, \ldots, \alpha_{a_{n}}^{(n)})\Big)_{n = 1}^{\infty}
\qquad\text{with}\qquad
\begin{array}{l}
%\text{(1) $\vec{X}^{(n)} = (x_{1}^{(n)}, x_{2}^{(n)}, \ldots, x_{1}^{(n)})$} \\
\text{(1) each $\alpha_{i}^{(n)} \in \mathcal{P}(H)_{n}$ , and} \\
\text{(2) $H$ is freely generated by $\vec{A}$.}
\end{array}
\]
Let
\[
\OPG(H) = \left\{ \text{ordered primitive generating sets of $H$} \right\}.
\]
\end{definition}

The following theorem uses the set $\OPG(H)$ to quantify the number of isomorphisms  $\mathcal{U}(\mathfrak{L}(\vec{a})) \cong H$.
%The second main result of the section is the following.

\begin{theorem}
\label{thm:OPGmaps}
Let $H$ be an \FGCCHA with $\vec{h} = \vecdim(H)$.  Then writing $\vec{a} = \phi_{h, a}(\vec{h})$ and $\vec{p} = \phi_{h, p}(\vec{h})$, we have bijections
\[
\begin{array}{rcl}
\Gamma: \{\text{graded Hopf algebra isomorphisms $\mathcal{U}(\mathfrak{L}(\vec{a})) \to H$}\} & \to & \OPG(H) \\
\phi & \mapsto & \phi(X_{\vec{a}}) \\
\end{array}
\]
and %\lucas{Change $A^{(n)}$ to $M^{(n)}$, in light of notation change for $\OPG(H)$}
\[
\begin{array}{rcl}
\Xi: \OPG(H) &\to& \left\{ \begin{array}{c}
\text{Sequences $\big(M^{(n)} \in \mathrm{Mat}_{a_{n} \times p_{n}}(\CC)\big)_{n=1}^{\infty}$ } \\
\text{with $\det\big( (M^{(n)}_{i, j})_{1 \le i, j \le a_{n}} \big) \neq 0$ for $n \ge 1$}
\end{array}\right\}
\end{array}.
\]
\end{theorem}

The proof of Theorem~\ref{thm:OPGmaps} is given in Section~\ref{sec:OPGproof}.  
As we do not explicitly describe our bijection, the second part of Theorem~\ref{thm:OPGmaps} should be read as a statement about how ``big'' the set $\OPG(H)$ is.  
Realizing our bijection requires one to make a rather complicated choice of basis for $\mathcal{P}(H)$, which we explain fully in the proof.  
However, we can illustrate the bijection concretely with a small example.  

\begin{example}
\label{ex:OPG}
We construct the set $\OPG(H)$ for the \FGCCHA $H = \CC\langle x, y, z\rangle$ freely generated by two primitive generators $x$ and $y$ in degree one and one primitive generator $z$ in degree two.  
In the notation of Theorem~\ref{thm:OPGmaps}, this means that $\vec{a} = (2, 1, 0, \ldots)$ and $\vec{p} = \phi_{a, p}(\vec{a}) = (2, 2, \ldots)$. 
The theorem then states that each element of $\OPG(H)$ is unqiely determined by the choice of a $2 \times 2$ invertible matrix $M^{(1)}$ and a $1 \times 2$ matrix $M^{(2)}$ with nonzero $1, 1$ entry.  
In particular, for any
\[
M^{(1)} = \begin{bmatrix}
a & b \\
c & d
\end{bmatrix}
\text{with $ac-bd \neq 0$}
\qquad\text{and}\qquad
M^{(2)} = \begin{bmatrix}
e & f
\end{bmatrix}
\text{with $e \neq 0$},
\]
the corresponding ordered primitive generating set is 
\[
\Big(\big(ax+by,\, cx+dy\big), \big(ez + f(xy-yx)\big), \emptyset, \ldots \Big) \in \OPG(H).
\]
\end{example}

\subsection{Proof of Theorem~\ref{thm:a sequence}}
\label{sec:classificationproof}

We will make use of the following established results
about free Lie algebras from~\cite{Reutenauer-FreeLieAlgebras}.

\begin{proposition}[{\cite[Theorem 0.5]{Reutenauer-FreeLieAlgebras}}]
\label{prop:LieToFree}
For a sequence $\vec{a}$ of a nonnegative integers, the enveloping algebra $\mathcal{U}(\mathfrak{L}(\vec{a}))$ is an \FGCCHA which is generated in degree $n$ by the $a_{n}$-many primitive generators of $\mathfrak{L}(\vec{a})$.
\end{proposition}

We now prove Theorem~\ref{thm:a sequence}; recall Proposition~\ref{prop:SequencesAsHopfStructures}.

\begin{proof}[Proof of Theorem~\ref{thm:a sequence}]
Given a sequence $\vec{a} \in \NN^{\ZZ_+}$, Proposition~\ref{prop:LieToFree} states that
$\mathcal{U}(\mathfrak{L}(\vec{a}))$ is an \FGCCHA freely generated by a set with graded
cardinality $\vec{a}$, so the map $\Phi$ is well defined.
Proposition~\ref{prop:LieToFree} also tells us that the map is injective: if $\vec{a}, \vec{b}  \in \NN^{\ZZ_+}$ have $a_{i} \neq b_{i}$ for some $i$, then the number of primitive generators of $\mathcal{U}(\mathfrak{L}(\vec{a}))$ and $\mathcal{U}(\mathfrak{L}(\vec{b}))$ must differ in degree $i$, and therefore $\mathcal{U}(\mathfrak{L}(\vec{a})) \not\cong \mathcal{U}(\mathfrak{L}(\vec{b}))$.

%By Proposition~\ref{prop:SequencesAsHopfStructures},
%\[
%\vecdim\big(\mathcal{U}(\mathfrak{L}(\vec{a}))\big) = \phi_{a, h}\big(\vec{a}),
%\]
%which by Aliniaeifard and Thiem's result (Theorem~\ref{thm:AT}) is an invariant of the isomorphism type.
%The map $\phi_{a, h}$ is invertible and hence injective,
%for any two sequences $\vec{a}, \vec{b}$ in the domain,
%\[
%\mathcal{U}(\mathfrak{L}(\vec{a})) \cong \mathcal{U}(\mathfrak{L}(\vec{b}))
%\qquad\text{if and only if}\qquad
%\vec{a} = \vec{b}.
%\]

To show that $\Phi$ is surjective, fix an isomorphism class $[H]$ of \FGCCHA with representative $H$ and let $\vec{h} = \vecdim(H)$.
By Proposition~\ref{prop:SequencesAsHopfStructures}, $\vec{a} = \phi_{h, a}(\vec{h})$ is the graded cardinality of the free generating set of $H$, and therefore $\vec{a} \in \mathbb{N}^{\mathbb{Z}_{+}}$; we now show that $\Phi(\vec{a}) = [H]$.  
By Proposition~\ref{prop:SequencesAsHopfStructures}, we have
\[
\vecdim(\mathcal{U}(\mathfrak{L}(\vec{a}))) = \phi_{a, h}( \phi_{h, a}(\vecdim(H))) = \vecdim{H}.
\]
Aliniaeifard and Thiem's result (Theorem~\ref{thm:AT}) now implies that $\mathcal{U}(\mathfrak{L}(\vec{a})) \cong H$, so that $\Phi(\vec{a}) = [\mathcal{U}(\mathfrak{L}(\vec{a}))] = [H]$ as desired.  
\end{proof}

\subsection{Proof of Theorem~\ref{thm:OPGmaps}}
\label{sec:OPGproof}

We now consider the problem of explicitly constructing isomorphisms between \FGCCHAs via the set $\OPG(H)$ defined at the beginning of the section.  Our proofs in this section and the next will make use of the subspaces
\[
H_{+} = \bigoplus_{n \geq 1} H_n
\qquad\text{and}\qquad 
H^{2}_{+} = \mu(H_{+} \otimes H_{+}).
\]
The space $H^{2}_{+}$ is a graded subspace of $H$ with
\begin{equation}
\label{eq:Hplusdef}
(H^{2}_{+})_{n} = \bigoplus_{\substack{\alpha \vDash n \\ \alpha \neq (n)}}
H_{\alpha_{1}} H_{\alpha_{2}}\cdots H_{\alpha_{\ell}} \subseteq H_{n}.
\end{equation}
By definition (see the beginning of Section~\ref{sec:HopfAlgebra}), the degree $n$ generators of $H$ descend to a basis of $H_{n}/ (H^{2}_{+})_{n}$, and therefore for any graded generating set $\vec{A} = \biguplus_{n \ge 1} \vec{A}^{(n)}$ of $H$, 
\[
\dim(H_{n}/(H^{2}_{+})_{n}) = |\vec{A}^{(n)}|.
\]
We also use an additional relation between $H_{+}/H^{2}_{+}$ and generators of $H$ which appears in~\cite{F23}.

\begin{proposition}[{\cite[Proposition 2.2 and 2.4]{F23}, \cite[Theorem 7.5]{MM65}}]
\label{prop:indecomposable generators}
If $H$ is a free graded connected algebra and $\vec{A} \subseteq H$ is a graded set with $H_{+} = \CC \vec{A} \oplus H_{+}^{2}$, then $\vec{A}$ freely generates $H$.  If moreover $H$ is a Hopf algebra and $A \subseteq \mathcal{P}(H)$, then $\mathcal{P}(H) = \CC \vec{A} \oplus [\mathcal{P}(H), \mathcal{P}(H)]$.
\end{proposition}

\begin{proof}[Proof of Theorem \ref{thm:OPGmaps}]
Let $\vec{a} = \phi_{h,a}(\vecdim(H))$ and $X_{\vec{a}}$ be an element of $\OPG(\mathcal{U}(\mathfrak{L}(\vec{a})))$.
We first consider the map 
\[
\begin{array}{rcl}
\Gamma: \{\text{graded Hopf algebra isomorphisms $\mathcal{U}(\mathfrak{L}(\vec{a})) \to H$}\} &\to& \OPG(H) \\
\phi & \mapsto & \phi(X_{\vec{a}}) \\
\end{array},
\]
where $\phi(X_{\vec{a}})$ is the  ordered set
\[
\phi(X_{\vec{a}}) = \biguplus_{n \ge 0} \big( \phi(x_{1}^{(n)}), \phi(x_{2}^{(n)}), \ldots, \phi(x_{a_{n}}^{(n)}) \big) \in \OPG(H).
\]
We show that $\Gamma$ is a bijection.  The first step is to observe that $\Gamma$
is well-defined: if $\phi: \mathcal{U}(\mathfrak{L}(\vec{a})) \to H$
is an isomorphism then the set $\phi(X_{\vec{a}})$ must freely generate $H$.
Moreover as $\phi$ is a graded isomorphism and $x_{i}^{(n)}$ is primitive
of degree $n$ for all $n \ge 1$ and $1 \le i \le a_{n}$, it follows that
$\phi(x_{i}^{(n)})$ is also primitive of degree $n$.
The injectivity of $\Gamma$ follows from the fact that any two algebra
morphisms which agree on a generating set must be equal.
For surjectivity, consider an arbitrary
$\vec{A} = \big((\alpha_{i}^{(n)})_{i=1}^{a_{n}}\big)_{n \ge 1} \in \OPG(H)$
and define $\phi_{\vec{A}}:\mathcal{U}(\mathfrak{L}(\vec{a})) \to H$ by
algebraically extending the mapping of generators $\phi_{\vec{A}}: x_{i}^{(n)} \mapsto \alpha_i^{(n)}$.
Then $\phi_{\vec{A}}$ is an isomorphism and $\phi_{\vec{A}}(X_{\vec{a}}) = \vec{A} \in \OPG(H)$.

We now define a bijection
\[
\begin{array}{rcl}
\Xi: \OPG(H) &\to& \left\{ \begin{array}{c}
\text{Sequences $\big(M^{(n)} \in \mathrm{Mat}_{a_{n} \times p_{n}}(\CC)\big)_{n=1}^{\infty}$ } \\
\text{with $\det\big( (M^{(n)}_{i, j})_{1 \le i, j \le a_{n}} \big) \neq 0$ for $n \ge 1$}
\end{array}\right\}
\end{array}.
\]
Our definition depends on a choice of $\vec{A} \in \OPG(H)$ and a homogeneous basis 
\[
\biguplus_{n \geq 1}\{v_1^{(n)},v_2^{(n)},\ldots,v_{p_n - a_n}^{(n)}\}
\]
of the derived subalgebra $[\mathcal{P}(H),\mathcal{P}(H)]$  (see Equation~\eqref{eq:derivedLiesubalgebra}), which we now fix.  
As $\vec{A} \subseteq \mathcal{P}(H)$, it follows from Proposition~\ref{prop:indecomposable generators} that $\mathcal{P}(H) = \CC \vec{A} \oplus [\mathcal{P}(H), \mathcal{P}(H)]$ so the set $\{\alpha_{1}^{(n)}, \ldots, \alpha_{a_{n}}^{(n)}, v_{1}^{(n)}, \ldots, v_{p_n - a_n}^{(n)}\}$ is a basis for the degree $n$ component of $\mathcal{P}(H)$.  
For any $\vec{B} = \big((\beta_{i}^{(n)})_{i=1}^{a_{n}}\big)_{n \ge 1} \in \OPG(H)$, we define $M^{(n)}$ to be the matrix whose rows express each $\beta_{i}^{(n)}$ in this basis, and set
\[
\Xi(\vec{B}) = (M^{(n)})_{n \ge 1}.
\]
The matrix $(M^{(n)}_{i, j})_{1 \le i, j \le a_{n}}$ gives the change of basis from $\vec{A}_{n}$ to $\vec{B}_{n}$ in the quotient $\mathcal{P}(H)_{n}/[\mathcal{P}(H),\mathcal{P}(H)]_{n}$.  Since this is invertible, it has a nonzero determinant.

%By \cite[Prop. 2.4]{F23}, $\mathcal{P}(H)$ is freely generated as a Lie algebra by $\vec{A}$, so that $\mathcal{P}(H)$ is isomorphic to the free Lie algebra $\mathfrak{L}(\vec{A})$ from Definition~\ref{defn:freelie}.  
%    Moreover, since $\mathfrak{L}(\vec{A}) = \CC \vec{A} \oplus [\mathfrak{L}(\vec{A}),\mathfrak{L}(\vec{A})]$, 
%    \[
%        \CC \vec{A} \cong \mathcal{P}(H)/[\mathcal{P}(H),\mathcal{P}(H)].
%    \]
%    We can then interpret an arbitrary choice of $\vec{S} \in \OPG(H)$ as some homogeneous change of basis of the quotient space $\mathcal{P}(H)/[\mathcal{P}(H),\mathcal{P}(H)]$ relative to $X$. \lucas{I'm not sure I see what this means.}
%    Each homogeneous component $\vec{S}^{(n)}$ then has the form
%    \[
%        \vec{S}^{(n)} = M^{(n)}_{0}\vec{A}^{(n)} + \vec{B}^{(n)},
%    \]
%    for some $M^{(n)}_{0} \in \GL(a_n,\CC)$ and $\vec{B}^{(n)} = \left(b_1^{(n)},b_2^{(n)},\ldots,b_{a_n}^{(n)}\right)$, for some $b_i^{(n)} \in [\mathcal{P}(H),\mathcal{P}(H)]_n$. 
%    
%    Let us fix . Then $\vec{B}^{(n)}$ can be identified with an $a_n$ by $p_n-a_n$ matrix $K^{(n)}$, such that $b_i^{(n)} = \sum_{j = 1}^{p_n - a_n}K_{ij}^{(n)}v_j^{(n)}$. Then we define the $a_n$ by $p_n$ matrix
%    \[
%    M^{(n)} = \Bigl[M^{(n)}_{0} \big| K^{(n)}\Bigr]
%    \]
%    as desired. 
%    
    On the other hand, given a sequence $\left(M^{(n)}\right)_{n \geq 1}$ with $M^{(n)} \in \textrm{Mat}_{a_n \times p_n}$ and
    \[\det\big( (M^{(n)}_{i, j})_{1 \le i, j \le a_{n}} \big) \neq 0,\] set
    \[
        \beta_i^{(n)} = \sum_{j = 1}^{a_n} M^{(n)}_{i, j}\alpha_{j}^{(n)} + \sum_{j = 1}^{p_n - a_{n}} M^{(n)}_{i, j+a_{n}}v_{j}^{(n)}.
    \]
    Then the ordered set
    \[
        \vec{B} = \biguplus_{n\geq 1} \left(\beta_1^{(n)},\beta_2^{(n)},\ldots,\beta_{a_n}^{(n)}\right)
    \]
    descends to a basis of $H_{+}/H_{+}^{2}$ and therefore generates $H$ by Proposition~\ref{prop:indecomposable generators}.  Since $\vec{B} \subseteq \mathcal{P}(H)$, $\vec{B} \in \OPG(H)$, and moreover $\Xi(\vec{B}) = (M^{(n)})_{n \ge 0}$.  Therefore we have constructed an inverse of $\Xi$, showing that it is a bijection.
%     and we see the correspondence between sequences
%     $\left(M^{(n)}\right)_{n \geq 1}$ and $\OPG(H)$ is surjective.
\end{proof}

\section{Classifying surjections between \FGCCHAs}
\label{sec:surjection}
We now consider the problem of constructing surjective homomorphisms between two \FGCCHAs, $H$ and $K$, beginning with a naive approach to this problem.  
Let $\vec{A} \in \OPG(H)$ and $\vec{B} \in \OPG(K)$ be ordered primitive generating sets as in Definition~\ref{def:OPG}, and let $\vec{a} = \veccard(\vec{A})$ and $\vec{b} = \veccard(\vec{B})$. 
If $\vec{a} \ge \vec{b}$, projecting from $\vec{A}$ to $\vec{B}$ gives a surjective homomorphism:
\begin{equation}
\label{eq:defaultsurjection}
\begin{array}{rcl}
f^{\vec{A}}_{\vec{B}}: H & \to & K \\[1ex]
\alpha^{(n)}_{i} & \mapsto & \begin{cases}
\beta^{(n)}_{i} & \text{for $1 \le i \le b_{n}$} \\
0 & \text{otherwise}
\end{cases}
\end{array}
\end{equation}
where $\alpha^{(n)}_{i}$ and $\beta^{(n)}_{i}$ denote the
$i$th element of the degree $n$ parts of $\vec{A}$ and $\vec{B}$, respectively.  
We find that this construction actually accounts for all surjections between \FGCCHAs.

\begin{theorem}
\label{thm:surjection}
Let $H$ and $K$ be \FGCCHAs which are generated by sets of graded cardinality $\vec{a}$ and $\vec{b}$, respectively.  
\begin{enumerate}
\item There is a surjective homomorphism $f: H \to K$ if and only if $\vec{a} \ge \vec{b}$.

\item For every surjective homomorphism $f: H \to K$, there exists a (nonunique) choice of
ordered primitive generating sets $\vec{A} \in \OPG(H)$ and $\vec{B} \in \OPG(K)$
for which $f = f^{\vec{A}}_{\vec{B}}$.

\end{enumerate}
\end{theorem}

The proof of Theorem~\ref{thm:surjection} is given after Lemma~\ref{lemma:primitive generators}, following some intermediate examples and results.  We also compare and contrast this result with a result of Aguiar--Lauve \cite{AL13} after Remark \ref{rem:isomorphism_maps}.

\begin{example}
\label{ex:quotientHopf}
Recall the \FGCCHA from Example~\ref{ex:OPG}, $H = \CC\langle x, y, z \rangle$.
Also let $K = \CC\langle u, v \rangle$ denote an \FGCCHA
which is freely generated by two degree-one primitive elements $u$ and $v$. 
By Theorem \ref{thm:surjection} (1), since $(2,1,0,\ldots) \geq (2,0,0,\ldots)$,
there is at least one surjective homomorphism from $H$ to $K$.  
We define $f: H \to K$ by
\[
f(x) = u,
\qquad
f(y) = v,
\qquad\text{and}\qquad
f(z) = uv-vu,
\]
which does not obviously match the format of Equation~\eqref{eq:defaultsurjection}.  However, if we write $f$ using a slightly different set of ordered primitive generators,
\[
\vec{A} = \Big(\big( x, y\big), \big( z - xy + yx \big), \emptyset, \ldots \Big)
\qquad
\text{and}
\qquad
\vec{B} =  \Big(\big( u, v\big), \emptyset, \ldots \Big),
\]
then we have 
\[
f(x) = u,
\qquad
f(y) = v,
\qquad\text{and}\qquad
f(z - xy + yx) = 0,
\]
so we have found the $\vec{A}$ and $\vec{B}$ for which $f = f^{\vec{A}}_{\vec{B}}$ as predicted by the theorem.
\end{example}

\begin{corollary}
Given a surjective homomorphism $f: H \to K$ between two \FGCCHAs,
\begin{enumerate}
\item $K$ is isomorphic to a Hopf subalgebra of $H$.

\item Every other surjective homomorphism $g: H \to K$ has the form
$g = \gamma \circ f \circ \tau$ for some automorphims $\gamma$ of $K$ and $\tau$ of $H$.

\end{enumerate}
\end{corollary}
\begin{proof}
Let $f$ be a surjective homomorphism from $H$ to $K$.  
By Theorem~\ref{thm:surjection}, there exist $\vec{A}  = \big( ( \alpha_{i}^{(n)})_{i = 1}^{a_{n}} \big)_{n \ge 0} \in \OPG(H)$ and $\vec{B}   = \big( ( \beta_{i}^{(n)})_{i = 1}^{b_{n}} \big)_{n \ge 0}  \in \OPG(K)$ such that $f = f^{\vec{A}}_{\vec{B}}$.  

To see (1), define a graded algebra homomorphism $\kappa: K \to H$ as the algebraic extension of $\beta^{(n)}_{i} \mapsto \alpha^{(n)}_{i}$ for all $n$ and $1\le i \le b_{n}$.  
Then $\kappa$ is the left inverse of $f^{\vec{A}}_{\vec{B}}$, so it is injective.  
We claim that  $\kappa$ respects coproducts, which implies that $\kappa(K) \subseteq H$ is a Hopf subalgebra; since $\kappa$ is injective $\kappa(K) \cong K$ and we will have proved (1).  
Since each $\beta^{(n)}_{i}$ and each $\alpha^{(n)}_{i}$ is primitive, it follows that
\[
(\kappa\otimes\kappa)(\Delta(\beta^{(n)}_{i}))
=
\kappa(\beta^{(n)}_{i}) \otimes \kappa(1) + \kappa(1) \otimes \kappa(\beta^{(n)}_{i})
= 
\alpha^{(n)}_{i} \otimes 1 + 1 \otimes \alpha^{(n)}_{i}
\]
and
\[
\Delta(\kappa(\beta^{(n)}_{i}))
=
\Delta(\alpha^{(n)}_{i})
= 
\alpha^{(n)}_{i} \otimes 1 + 1 \otimes \alpha^{(n)}_{i}
=
(\kappa\otimes\kappa)(\Delta(\beta^{(n)}_{i}))
\]
As $\Delta$ is an algebra homomorphism and $K$ and $H$ are freely generated by $\vec{B}$ and $\vec{A}$ as algebras, the above equation implies that $(\kappa\otimes\kappa) \circ \Delta = \Delta \circ \kappa$ as desired, completing the proof of (1).

For (2),  let $g$ be a surjective homomorphism from $H$ to $K$.  
We first note that by Theorem~\ref{thm:surjection},  $g = f^{\vec{C}}_{\vec{D}}$ for some $\vec{C} \in \OPG(H)$ and $\vec{D} \in \OPG(K)$.  We then take $\tau: H \to H$ to be the extension of the order-preserving map which sends $\vec{C}$ to $\vec{A}$, and similarly $\gamma: K \to K$ as the extension of the order-preserving mapping of $\vec{B}$ to $\vec{D}$.  Then $g = \gamma \circ f \circ \tau$.
\end{proof}

Our proof makes use of the \emph{Eulerian idempotent}, also known as the canonical projection or the first Eulerian idempotent (see \cite{P94}, \cite[\S 1.4]{AS05cc} or \cite[\S 4.5.2]{loday2013cyclic}).  This is the map
%\cite[Section 1.4]{AL15},\cite{AS05cc} ,
%\cite[\S 3]{GERSTENHABER1991263},\cite[\S 4.5.2]{loday2013cyclic},\cite[\S 3.2]{Reutenauer-FreeLieAlgebras}, )
\begin{equation}
\label{eq:eulerianidempotent}
\mathbf{e}: H \to H
\qquad\text{such that}\qquad
\mathbf{e}(x) = x - \frac{1}{2}\mu \circ \Delta_{+}(x) + \frac{1}{3} \mu^{(2)} \circ \Delta_{+}^{(2)}(x) - \cdots 
\end{equation}
where $\Delta_{+}$ denotes the positive coproduct $\Delta_{+}(x) = \Delta(x) - x \otimes 1 - 1 \otimes x$,
and we write $\mu^{(k)}$ and $\Delta^{(k)}_{+}$ for the $k$-fold self-composition of $\mu$ and $\Delta_{+}$, respectively.  
By associativity and coassociativity of the Hopf algebra, these maps are well defined, and further $\Delta_{+}^{(n)}$ is zero on all elements of degree greater than $n$, so the sum in~\eqref{eq:eulerianidempotent} is always finite.

When $H$ is graded, connected and cocommutative, the Eulerian idempotent is a linear projection onto $\mathcal{P}(H)$~\cite[Theorem 9.4]{S94}.
Following~\cite{P99}, we call the family of all graded linear projections from $H$ onto $\mathcal{P}(H)$ \emph{Lie idempotents}.

Aguiar and Sottile use the Eulerian idempotent to produce primitive generators for the Grossman-Larson
Hopf algebra of heap-ordered trees in~\cite{AS05cc},
as do Novelli and Thibon for a Hopf algebra with dimension sequence given by the Catalan numbers in~\cite[\S 5]{NT05}.
Lauve and Mastnak use a related primitive idempotent to produce primitive generators for
the symmetric functions in noncommuting variables in~\cite{LM11}.  
Foissy and Patras recently showed that this method holds more generally for any \FGCCHA and Lie idempotent~\cite{FP24}, 
generalizing the following result of Patras and Reutenauer~\cite{PR04}.
%This method holds more generally for any \FGCCHA and was previously stated in the following result from \cite{PR04}.
%\lucas{Add something about Foissy's paper here!}

\begin{lemma}[{\cite[Lemma 22]{PR04}}]
\label{lemma:primitive generators}
Let $H$ be an \FGCCHA freely generated by a graded set $X = \biguplus_{n \geq 1} X_{n}$.  
Then the graded set
\[
\mathbf{e}(X) = \biguplus_{n \ge 1} \{\mathbf{e}(x) \;|\; x \in X_{n}\}
\]
is a complete set of primitive generators of $H$.
\end{lemma}
%\eric{I found a paper which proves this exactly - Lemma 22 of Patras, Reutenauer - On Descent Algebras and Twisted Bialgebras \url{https://reutenauer.math.uqam.ca/wp-content/uploads/2024/05/Twisted-bialgebras-compresse.pdf} }
%\begin{proof}
%    For all $k \geq 2$ and $x \in X$, the term 
%    \[
%    \frac{(-1)}{k}^{k-1}\mu^{(k-1)} \circ \Delta_{+}^{(k-1)}(x)
%    \]
%    is in the space $H^{+2}$. Then we have
%    \[
%    \mathbf{e}(x) = x + \sum_{k \geq 2}\frac{(-1)}{k}^{k-1}\mu^{(k-1)} \circ \Delta_{+}^{(k-1)}(x),
%    \]
%    hence $\mathbf{e}(x) + H^{+2} = x + H^{+2}$. Because $X$ descends to a basis for $H^+/H^{+2}$, so must $\mathbf{e}(X)$, and by Proposition \ref{prop:indecomposable generators}, $\mathbf{e}(X)$ freely generates $H$.
%\end{proof}

\begin{proof}[Proof of Theorem~\ref{thm:surjection}]
Let $H$ and $K$ be \FGCCHAs which are generated by sets of graded cardinality $\vec{a}$ and $\vec{b}$, respectively.  
If $\vec{a} \ge \vec{b}$, then Equation~\eqref{eq:defaultsurjection} defines a surjective homomorphism between $H$ and $K$, giving the ``if'' part of (1).

To prove the converse and (2), suppose that we have surjective homomorphism
\[
f: H \to K.
\]
For the next part of the argument, fix $n \ge 1$.  
Recall the subspaces $(H_{+}^{2})_{n} \subseteq H_{n}$ and $(K_{+}^{2})_{n} \subseteq K_{n}$ defined in Equation~\eqref{eq:Hplusdef}, and that 
\[
\dim(H_{n} / (H^{2}_{+})_{n}) = a_{n}
\qquad\text{and}\qquad
\dim(K_{n} / (K^{2}_{+})_{n}) = b_{n}.
\]
As $f$ is a graded surjective algebra homomorphism, it must map $H_{n}$ onto $K_{n}$ and $(H^{2}_{+})_{n}$ onto $(K^{2}_{+})_{n}$.  
Therefore $f$ induces a surjective linear transformation from $H_{n} / (H_{+}^{2})_{n}$ to $K_{n} / (K_{+}^{2})_{n}$.  
Surjectivity implies that $a_{n} = \dim(H_{n} / (H_{+}^{2})_{n}) \ge \dim(K_{n} / (K_{+}^{2})_{n})  = b_{n}$.

%Dimension considerations now give that $a_{n} \ge b_{n}$, so the ``only if'' part of (1) holds.  

We continue with our fixed $n$ for (2).  
First choose a tuple of elements $\big( y^{(n)}_{1}, \ldots, y^{(n)}_{b_{n}} \big)$ from $K_{n}$ which descend to a basis of $K_{n} / (K_{+}^{2})_{n}$.
Let $\beta_{i}^{(n)} = \mathsf{e}(y_{i}^{(n)})$ where $\mathsf{e}$ is the 
Eulerian idempotent. 
By Lemma~\ref{lemma:primitive generators} and Proposition~\ref{prop:indecomposable generators} the tuple $\big( \beta^{(n)}_{1}, \ldots, \beta^{(n)}_{b_{n}} \big)$ also descends to a basis of $K_{n} / (K^{2}_{+})_{n}$.  
Since $f$ induces a surjection $H_{n} / (H^{2}_{+})_{n} \to K_{n} / (K^{2}_{+})_{n}$, we may choose an ordered subset $\big( x^{(n)}_{1}, \ldots, x^{(n)}_{a_{n}} \big)$ of $H_{n}$ which descends to a basis of $H_{n}/(H^{2}_{+})_{n}$ and moreover has the property that
\[
f(x^{(n)}_{i}) = \begin{cases} \beta^{(n)}_{i} & \text{if $i \le b_{n}$} \\ 0 & \text{otherwise.} \end{cases}
\]
Now take $\alpha_{i}^{(n)} = \mathsf{e}(x_{i}^{(n)})$. 
By the argument above the set $\big( \alpha^{(n)}_{1}, \ldots, \alpha^{(n)}_{a_{n}} \big)$ is a basis of $H_{n} / (H^{2}_{+})_{n}$.
Moreover $f$ commutes with $\mathsf{e}$: $f \circ \mu^{(k)} = \mu^{(k)} \circ f^{\otimes k}$ because $f$ is a Hopf morphism, and similarly $f^{\otimes k} \circ \Delta_{+} = \Delta_{+} \circ f$ because
\[
\Delta_{+} \circ f
= \Delta \circ f - f \otimes 1 - 1 \otimes f
= (f\otimes f)\circ \Delta - f \otimes 1 - 1 \otimes f
= (f\otimes f)\circ \Delta_{+},
\]
as $f(1) = 1$.  Therefore,
\[
f(\alpha_{i}^{(n)}) = f( \mathsf{e}(x_{i}^{(n)})) = \mathsf{e}(f( x_{i}^{(n)})) = \mathsf{e}(\beta_{i}^{(n)}) = \beta_{i}^{(n)}.
\]

Finally, repeat this construction for all $n \ge 1$.  By Proposition~\ref{prop:indecomposable generators} the sequences 
\[
\vec{A} = \big( (\alpha_{1}^{(n)}, \ldots, \alpha_{a_{n}}^{(n)}) \big)_{n \ge 1}
\qquad\text{and}\qquad
\vec{B} = \big( (\beta_{1}^{(n)}, \ldots, \beta_{b_{n}}^{(n)}) \big)_{n \ge 1}
\]
belong to $\OPG(H)$ and $\OPG(K)$ respectively, and by construction $f = f^{\vec{A}}_{\vec{B}}$.
\end{proof}

\section{Classifying Hopf subalgebras by primitive dimensions}
\label{sec:subalgebras}

Let $H$ be an \FGCCHA.  
In this section we construct all the isomorphism classes of Hopf subalgebras of $H$; in particular we show that every Hopf subalgebra is an \FGCCHA satisfying a simple condition.
To begin, recall the functions $\phi_{a, p}$ and $\phi_{h, a}$ from Definition~\ref{defn:sequencetransfer}, and recall that for any nonnegative sequence $\vec{a} = (a_{1}, a_{2}, \ldots)$, we write $\mathfrak{L}(\vec{a})$ for the free Lie algebra generated by the graded set $X_{\vec{a}}$ with $a_{n}$ elements in degree $n$ as in Equation~\eqref{eq:La}.

\begin{theorem}
\label{thm:subclassification2}
Let $H$ be a free graded connected cocommutative Hopf algebra and let $\vec{p} = \vecdim(\mathcal{P}(H))$.  Then there is a bijection
\[
\begin{array}{rcl}
\left\{\begin{array}{c}
\text{Isomorphism classes of} \\
\text{Hopf subalgebras $K \subseteq H$}
\end{array}\right\}
& \leftrightarrow & 
\left\{\begin{array}{c}
\text{Sequences $\vec{b} \in \NN^{\ZZ_{+}}$} \\
\text{for which $\phi_{a, p}(\vec{b}) \le \vec{p}$} 
\end{array}\right\} \\[1.5em]
K & \mapsto & \phi_{h, a} \big( \vecdim(K)\big) \\
\mathcal{U}(\mathfrak{L}(\vec{b}))) & \mapsfrom & \vec{b}
\end{array}
\]
\end{theorem}

The proof of Theorem~\ref{thm:subclassification2} will construct,
for each sequence $\vec{b}$, a Hopf subalgebra in $H$ isomorphic to
$\mathcal{U}(\mathfrak{L}(\vec{b}))$.  The proof of this result appears after some intermediate results
and motivating examples.

\begin{remark}\label{rem:isomorphism_maps}
Using the sequence transfer maps from Definition~\ref{defn:sequencetransfer}, Theorem~\ref{thm:subclassification2} implies that isomorphisms of Hopf subalgebras of $H$ are also indexed by sequences $\vec{q} \in \NN^{\ZZ_{+}}$ for which $\vec{p} \ge \vec{q}$ and $\phi_{p, a}(\vec{q}) \in \NN^{\ZZ_{+}}$.
\end{remark}

In~\cite{AL13}, Aguiar and Lauve give an analogue of Lagrange's theorem for Hopf algebras,
which implies a necessary condition for the existence of a Hopf
subalgebra with a given dimension sequence.
Specifically,~\cite[Corollary 1.4]{AL13} states that if $H$ is a graded connected Hopf algebra with
a graded connected Hopf subalgebra $K \subseteq H$, then the ratio of Poincar\'e series
\begin{equation}
\label{eq:ratio_gfs}
\frac{1 + \sum_{n \geq 1} h_n t^n}{1 + \sum_{n \geq 1} k_n t^n}
\qquad
\text{where $\vec{h} = \vecdim(H)$ and $\vec{k} = \vecdim(K)$}
\end{equation}
has non-negative coefficients as a series in $t$.
Expanding Equation~\eqref{eq:ratio_gfs}, this condition is equivalent to 
\[
\sum_{r=0}^n \sum_{\alpha \models n-r} (-1)^{\ell(\alpha)} h_r k_\alpha \geq 0
\qquad\text{for all $n \geq 1$.}
\]

The result from \cite[Corollary 1.4]{AL13} (much like Lagrange's theorem for groups)
is only a necessary condition on the dimension sequences of a Hopf algebra
and subalgebra while Theorem~\ref{thm:subclassification2}
also includes sufficient conditions.
This can be illustrated by example: let $\vec{p} = \phi_{h,p}(\vec{h})$ and
$\vec{q} = \phi_{h,p}(\vec{k})$.
Equation \eqref{eq:gf_relation} implies that the ratio in Equation \eqref{eq:ratio_gfs} is equal to
\[
\prod_{d \geq 1} \frac{1}{(1-t^d)^{p_d - q_d}}~.
\]
We note in the following example a pair of sequences whose ratio of generating
functions is positive, but there does not exist a Hopf subalgebra with the given
dimension sequence.  It also illustrates
a counterintuitive aspect of containment between \FGCCHAs.

\begin{example}
Consider the primitive dimension sequence
$\vec{p} = (1,0,0,\ldots)$ and set $\vec{a} = \phi_{p,a}(\vec{p}) = \vec{p}$
and $\vec{h} = \phi_{p,h}(\vec{p})$ so that for $n \geq 1$,
$h_n$ is equal to 1.
Next set, 
$\vec{q} = (0,1,0,0,\ldots)$ and set $\vec{b} = \phi_{p,a}(\vec{q}) = \vec{q}$ and
$\vec{k} = \phi_{p,h}(\vec{q})$ so that $k_n$ is equal to $1$
if $n$ is even and $0$ otherwise.  Moreover, the ratio of the generating functions
in Equation \eqref{eq:ratio_gfs} is equal to $1+t$.

It is the case that all of these coefficients are non-negative, however by
Theorem \ref{thm:subclassification2}, there does not exist a Hopf subalgebra
of $H := \Phi(\vec{a})$ having dimension sequence $\vec{k}$.

We note that curiously if $\vec{a}=(2,0,0,\ldots)$ and $\vec{b}=(0,1,0, 0,\ldots)$ then
we do have containment of the \FGCCHAs.
\end{example}

\begin{example}
\label{ex:NSymInCxy}
Recall the \FGCCHA $K = \CC\langle u, v \rangle$ from Example~\ref{ex:quotientHopf}, so that $u$ and $v$ have degree one.  Then $\vec{p} = \vecdim(\mathcal{P}(K))$ is the sequence
$\vec{p} = (2, 1, 2, 3, 6, 9, 18, 30,\ldots)$~\cite[\href{https://oeis.org/A001037}{A001037}]{OEIS} given by 
\[
p_{n} = \frac{1}{n} \sum_{d | n} \mu(n/d) 2^{d}.
\]

Then we can take 
\[
\vec{q} = \vec{p} - (1, 0, 0, \ldots)
\]
so that clearly $\vec{p} \ge \vec{q}$.  
The sequence $\vec{q}$ is identified in~\cite[\href{https://oeis.org/A059966}{A059966}]{OEIS}, and direct computation then gives that
\[
\phi_{p, a}(\vec{q}) = (1, 1, 1, \ldots) \in \NN^{\ZZ_{+}},
\]
so this sequence determines a subalgebra isomorphic to $\mathcal{U}(\mathfrak{L}(1, 1, 1, \ldots))$ inside of $K$.   
In Example~\ref{ex:NSym} we identify $\mathcal{U}(\mathfrak{L}(1, 1, 1, \ldots)) \cong \mathsf{NSym}$, so $K$ has a Hopf subalgebra isomorphic to $\mathsf{NSym}$.
\end{example}

We now state and prove an intermediate result in the proof of Theorem~\ref{thm:subclassification2}; the proof of the theorem follows.  Recall the definition of the derived subalgebra $[L, L]$ of a graded Lie algebra $L$ from Equation~\eqref{eq:derivedLiesubalgebra}.

\begin{lemma}
\label{lem:deriveddimension}
Let $\vec{a} = (a_{1}, a_{2}, \ldots) \in \NN^{\ZZ_{+}}$.  Then
\[
\vecdim([\mathfrak{L}(\vec{a}), \mathfrak{L}(\vec{a})]) = \phi_{a, p}(\vec{a}) - \vec{a},
\]
or equivalently $\vec{a} = \vecdim\big(\mathfrak{L}(\vec{a})\big/ [\mathfrak{L}(\vec{a}), \mathfrak{L}(\vec{a})]\big)$.
\end{lemma}
\begin{proof}
Using Proposition~\ref{prop:indecomposable generators}, $ \mathfrak{L}(\vec{a}) = \CC X_{\vec{a}} \oplus [\mathfrak{L}(\vec{a}), \mathfrak{L}(\vec{a})]$.  Taking the dimension sequence of each space, we obtain an equivalent equation.
\end{proof}

\begin{lemma}[Shirshov--Witt Theorem~\cite{S09, Witt56}]
\label{lem:ShiWitt}
Every Lie subalgebra of a free Lie algebra is free.
\end{lemma}

\begin{proof}[Proof of Theorem~\ref{thm:subclassification2}]
We first show the map $K \mapsto \phi_{h, a} \big( \vecdim(K)\big)$ is injective on isomorphism classes.  Suppose that $J$ and $K$ are Hopf subalgebras of $H$ with the property that $\phi_{h, a} \big(\vecdim(J)) = \phi_{h, a} \big(\vecdim(K))$.  
Since $\phi_{h, a}$ is invertible by Proposition~\ref{prop:sequences} $\vecdim(J) = \vecdim(K)$.  We therefore conclude by Theorem~\ref{thm:AT} that $J \cong K$.

Now suppose that $\vec{b} \in \NN^{\ZZ_{+}}$ has the property that $\vec{q} = \phi_{a, p}(\vec{b}) \le \vec{p}$.  
We will construct a Hopf subalgebra $K$ of $H$ which is isomorphic to $\mathcal{U}(\mathfrak{L}(\vec{b}))$; by Proposition~\ref{prop:sequences} this will show that the  map $K \mapsto \phi_{h, a} \big( \vecdim(K)\big)$ is surjective, completing the proof.

There is some subtlety to defining $K$.  Since we have not assumed that $a_{n} \ge b_{n}$, we cannot take $K$ to be generated by some subset of the free generators of $H$, and may have to choose elements of $H^{2}_{+}$ as generators for $K$.  
These choices cannot always be expressed uniformly, so we resort to an inductive argument to demonstrate their existence.  
Specifically, we construct a tower of Lie subalgebras 
\[
0 = L^{(1)} \subseteq L^{(2)} \subseteq L^{(3)} \subseteq \cdots \subseteq \mathcal{P}(H)
\]
such that 
\[
L^{(n)} \cong \mathfrak{L}(b_{1}, \ldots, b_{n-1}, 0, 0, \ldots).
\]
Taking $L^{(\infty)}$ to be the union of the $L^{(n)}$, we obtain $K$ as $\mathcal{U}(L^{(\infty)}) \subseteq \mathcal{U}(\mathcal{P}(H)) = H$.

We proceed by induction with base case $L^{(1)} = 0$.  
We then assume inductively that $L^{(n)}$ has been constructed.  
By Lemma~\ref{lem:deriveddimension} and our inductive hypothesis,
\[
\vecdim([L^{(n)}, L^{(n)}])
 = 
\phi_{a, p}(b_{1}, \ldots, b_{n-1}, 0, 0, \ldots) - (b_{1}, \ldots, b_{n-1}, 0, 0, \ldots).
\]
Thus $\vecdim([L^{(n)}, L^{(n)}])_{n}$ is equal to the $n$th term of $\phi_{a, p}(b_{1}, \ldots, b_{n-1}, 0, 0, \ldots)$, which is also equal to $\vecdim(L^{(n)})_{n}$.  
We can further deduce, after a careful examination of the definition of $\phi_{a, p}$, that
\begin{equation}
\label{eq:inductivedimensionequality}
\phi_{a, p}(b_{1}, \ldots, b_{n-1}, 0, 0, \ldots)_{n}
=
\phi_{a, p}(\vec{b})_{n} - b_{n}
=
q_{n} - b_{n}.
\end{equation}
Therefore,
\[
\vecdim\left( \mathcal{P}(H) \big/ [L^{(n)}, L^{(n)}] \right)_{n}
= p_{n} - (q_{n} - b_{n})
= b_{n} + (p_{n} - q_{n})
\ge b_{n}.
\]

The preceding inequality shows that we can choose $b_{n}$ linearly independent elements $\{x^{(n)}_{1}, \ldots, x^{(n)}_{b_{n}}\}$ of $\mathcal{P}(H)_{n}$ which remain linearly independent modulo $[L^{(n)}, L^{(n)}]$ and define $L^{(n+1)}$ to be the Lie algebra generated by  elements of $L^{(n)}$ as well as $x^{(n)}_{1}, \ldots, x^{(n)}_{b_{n}}$.

%\[
%L^{(n+1)} = \langle L^{(n)}, x^{(n)}_{1}, \ldots, x^{(n)}_{b_{n}} \rangle.
%\]
Since $L^{(n+1)}$ is a Lie subalgebra of a free Lie algebra, namely $\mathcal{P}(H)$, Lemma~\ref{lem:ShiWitt} states that $L^{(n+1)}$ is also free.  
As $L^{(n+1)}$ is generated by any generating set of $L^{(n)}$ and $b_{n}$-many homogeneous elements of degree $n$, we deduce that
\[
L^{(n+1)} \cong \mathfrak{L}(b_{1}, \ldots, b_{n-1}, b_{n}', 0, \ldots )
\qquad\text{for some $b_{n}' \le b_{n}$}.
\]
The final step of the proof is to show that $b_{n}' = b_{n}$.  Using Lemma~\ref{lem:deriveddimension} and the definition of $L^{(n+1)}$,
\begin{align*}
b_{n}' &= \dim\left( L^{(n+1)}_{n} \big/ [L^{(n+1)}, L^{(n+1)}]_{n}\right)% \\
%&= \dim\left(  \CC\operatorname{-span}\{x^{(n)}_{1}, \ldots, x^{(n)}_{b_{n}}\} \oplus [L^{(n)}, L^{(n)}]_{n} \big/ [L^{(n)}, L^{(n)}]_{n} \right) \\
%&= b_{n}. \qedhere
\end{align*}
Then $[L^{(n+1)}, L^{(n+1)}]_{n} = [L^{(n)}, L^{(n)}]_{n}$ as $L^{(n)}$ and $L^{(n+1)}$ contain the same homogeneous elements of degree $n - 1$ and less.  Therefore,
\begin{align*}
 \dim\!\left( L^{(n+1)}_{n} \big/ [L^{(n+1)}, L^{(n+1)}]_{n}\right) 
&= \dim\!\left(  \CC\operatorname{-span}\{x^{(n)}_{1}, \ldots, x^{(n)}_{b_{n}}\} \oplus [L^{(n)}, L^{(n)}]_{n} \big/ [L^{(n)}, L^{(n)}]_{n} \right) \\
&= b_{n}.
\end{align*}
\end{proof}

\bibliographystyle{plain}
\bibliography{bibliography}{}

@article{AT20,
  title={Pattern groups and a poset based {H}opf monoid},
  author={Aliniaeifard, Farid and Thiem, Nathaniel},
  journal={Journal of Combinatorial Theory, Series A},
  volume={172},
  pages={105187},
  year={2020},
  publisher={Elsevier}
}

@article{AT22,
  title={Hopf Structures in the Representation Theory of Direct Products},
  author={Aliniaeifard, Farid and Thiem, Nathaniel},
  journal={Electronic Journal of Combinatorics},
  volume={29},
  pages={Paper No. 4.39},
  year={2022},
  publisher={Elsevier}
}

@article{RS06,
  title={Symmetric functions in noncommuting variables},
  author={Rosas, Mercedes and Sagan, Bruce},
  journal={Transactions of the American Mathematical Society},
  volume={358},
  number={1},
  pages={215--232},
  year={2006}
}

@article{BHRZ05,
  title={Grothendieck Bialgebras, Partition Lattices, and Symmetric Functions in Noncommutative Variables},
  author={Bergeron, Nantel and Hohlweg, Christophe and Rosas, Mercedes and Zabrocki, Mike},
  journal={The Electronic Journal of Combinatorics},
  pages={R75--R75},
  year={2006}
}

@article{NT05,
  title={Hopf algebras and dendriform structures arising from parking functions},
  author={Novelli, Jean-Christophe and Thibon, Jean-Yves},
  journal={Fundamenta Mathematicae},
  volume={193},
  number={1},
  pages={189--241},
  year={2007}
}

@article{F12,
  title={Free and cofree {H}opf algebras},
  author={Foissy, Lo{\"\i}c},
  journal={Journal of Pure and Applied Algebra},
  volume={216},
  number={2},
  pages={480--494},
  year={2012},
  publisher={Elsevier}
}

@unpublished{F23,
  title={Primitive elements of a connected free bialgebra},
  author={Foissy, Lo{\"\i}c},
  note={arXiv preprint \arxiv{2309.16199}},
  year={2023}
}

@article{LM11,
  title={The primitives and antipode in the {H}opf algebra of symmetric functions in noncommuting variables},
  author={Lauve, Aaron and Mastnak, Mitja},
  journal={Advances in Applied Mathematics},
  volume={47},
  number={3},
  pages={536--544},
  year={2011},
  publisher={Elsevier}
}

@article{AS05,
title = {Structure of the {M}alvenuto–{R}eutenauer {H}opf algebra of permutations},
journal = {Advances in Mathematics},
volume = {191},
number = {2},
pages = {225-275},
year = {2005},
issn = {0001-8708},
doi = {https://doi.org/10.1016/j.aim.2004.03.007},
url = {https://www.sciencedirect.com/science/article/pii/S0001870804000787},
author = {Marcelo Aguiar and Frank Sottile}
}

@article{AL15,
  title={The characteristic polynomial of the {A}dams operators on graded connected {H}opf algebras},
  author={Aguiar, Marcelo and Lauve, Aaron},
  journal={Algebra \& Number Theory},
  volume={9},
  number={3},
  pages={547--583},
  year={2015},
  publisher={Mathematical Sciences Publishers}
}

@article{W36,
  title={Symmetric functions of non-commutative elements},
  author={Wolf, Margarete},
  journal={Duke Math. J.},
  volume={2},
  number={1},
  pages={626--637},
  year={1936}
}

@article{MM65,
  title={On the Structure of {H}opf Algebras},
  author={Milnor, John W. and Moore, John C.},
  journal={Annals of Mathematics},
  volume={81},
  pages={211--264},
  year={1965},
  publisher={Mathematics Department, Princeton University}
}

@article{KK95,
  title={Free {L}ie Algebras, Generalized {W}itt Formula, and the Denominator Identity},
  author={Kang, Seok-Jim and Kim, Myung-Hwan},
  journal={Journal of Algebra},
  volume={183},
  pages={560--594},
  year={1995},
  publisher={Elsevier}
}

@misc{OEIS,
   title = {The On-Line Encyclopedia of Integer Sequences},
   author = {OEIS Foundation Inc.},
   year ={2024},
   note = {\url{https://oeis.org}}
}

@article{S94,
  title={Incidence {H}opf Algebras},
  author={Schmitt, William R.},
  journal={Journal of Pure and Applied Algebra},
  volume={96},
  pages={299--330},
  year={1994},
  publisher={Elsevier}
}

@article{Block85,
title = {Commutative {H}opf algebras, {L}ie coalgebras, and divided powers},
journal = {Journal of Algebra},
volume = {96},
number = {1},
pages = {275-306},
year = {1985},
issn = {0021-8693},
doi = {https://doi.org/10.1016/0021-8693(85)90050-X},
url = {https://www.sciencedirect.com/science/article/pii/002186938590050X},
author = {Richard E. Block}
}

@Inbook{S09,
    author="Shirshov, Anatoly I.",
    title={Subalgebras of Free {L}ie Algebras},
    bookTitle="Selected Works of A. I. Shirshov",
    year="2009",
    publisher="Birkh{\"a}user Basel",
    address="Basel",
    pages="3--13",
    isbn="978-3-7643-8858-4",
    doi="10.1007/978-3-7643-8858-4_1",
    url="https://doi.org/10.1007/978-3-7643-8858-4_1"
}

@book{Reutenauer-FreeLieAlgebras,
  title={Free {L}ie Algebras},
  author={Reutenauer, Christophe},
  isbn={0-19-853679-8},
  series={London Mathematical Society Monographs New Series},
  year={1993},
  publisher={Oxford University Press}
}

@article{GKLLRT,
  title={Noncommutative Symmetric Functions},
  author={Gelfand, Israel M. and Krob, Daniel and Lascoux, Alain and Leclerc, Bernard and Retakh, Vladimir S. and Thibon, Jean-Yves},
  journal={Advances in Mathematics},
  volume={112},
  number={2},
  pages={218--348},
  year={1995},
  publisher={Elsevier BV}
}

@article{H07,
  title={The primitives of the {H}opf algebra of noncommutative symmetric functions},
  author={Hazewinkel, Michiel},
  journal={São Paulo Journal of Mathematical Sciences},
  volume={1},
  url={https://www.revistas.usp.br/spjournal/article/view/7},
  DOI={10.11606/issn.2316-9028.v1i2p175-202},
  year={2007},
  month={dez.},
  pages={175–202} }

@article{Bergeron_2002,
	title={Shifted quasi-symmetric functions and the {H}opf algebra of peak functions},
	volume={246},
	ISSN={0012-365X},
	url={http://dx.doi.org/10.1016/S0012-365X(01)00251-5},
	DOI={10.1016/s0012-365x(01)00251-5},
	number={1–3},
	journal={Discrete Mathematics},
	publisher={Elsevier BV},
	author={Bergeron, Nantel and Mykytiuk, Stefan and Sottile, Frank and van Willigenburg, Stephanie},
	year={2002},
	month=mar, pages={57–66} }

@book{loday2013cyclic,
	title={Cyclic Homology},
	author={Loday, Jean-Louis},
	isbn={9783662113899},
	lccn={97022418},
	series={Grundlehren der mathematischen Wissenschaften},
	url={https://books.google.ca/books?id=RzHqCAAAQBAJ},
	year={2013},
	publisher={Springer Berlin Heidelberg}
}

@article{AS05cc,
	title={Cocommutative {H}opf algebras of permutations and trees},
	author={Aguiar, Marcelo and Sottile, Frank},
	journal={Journal of Algebraic Combinatorics},
	volume={22},
	number={4},
	pages={451--470},
	year={2005},
	publisher={Springer}
}

@article{PR04,
	title={On Descent Algebras and Twisted Bialgebras},
	author={Patras, Fr\'ed\'eric and Reutenauer, Christophe},
	journal={Moscow Mathematical Journal},
	volume={4},
	number={1},
	pages={199--216},
	year={2004},
	note={\href{https://reutenauer.math.uqam.ca/wp-content/uploads/2024/05/Twisted-bialgebras-compresse.pdf}{Author webpage link}},
	publisher={Independent University of Moscow}
}

@article{GL09,
  title={Hopf algebras of heap ordered trees and permutations},
  author={Grossman, Robert L. and Larson, Richard G.},
  journal={Communications in Algebra},
  volume={37},
  number={2},
  pages={453--459},
  year={2009},
  publisher={Taylor \& Francis}
}

@article{Li15,
  title={The monomial basis and the {$Q$}-basis of the {H}opf algebra of parking functions},
  author={Li, Teresa Xueshan},
  journal={Journal of Algebraic Combinatorics},
  volume={42},
  number={2},
  pages={473--496},
  year={2015},
  publisher={Springer}
}

@article{HNT08,
  title={Commutative combinatorial {H}opf algebras},
  author={Hivert, Florent and Novelli, Jean-Christophe and Thibon, Jean-Yves},
  journal={Journal of Algebraic Combinatorics},
  volume={28},
  pages={65--95},
  year={2008},
  publisher={Springer}
}

@unpublished{FP24,
  TITLE = {{Lie theory of free cocommutative and commutative cofree {H}opf algebras}},
  AUTHOR = {Foissy, Lo{\"i}c and Patras, Fr{\'e}d{\'e}ric},
  URL = {https://hal.science/hal-04773293},
  NOTE = {HAL preprint \href{https://hal.science/hal-04773293v1}{hal-04773293v1}},
  YEAR = {2024},
  url = {https://hal.science/hal-04773293v1},
  HAL_ID = {hal-04773293},
  HAL_VERSION = {v1},
}

@article{Witt56,
  title={Die Unterringe der freien Lieschen Ringe},
  author={Witt, Ernst},
  journal={Mathematische Zeitschrift},
  volume={64},
  number={1},
  pages={195--216},
  year={1956},
  publisher={Springer}
}

@Article{AgEtAl12,
 Author = {Aguiar, Marcelo and Andr{\'e}, Carlos and Benedetti, Carolina and Bergeron, Nantel and Chen, Zhi and Diaconis, Persi and Hendrickson, Anders and Hsiao, Samuel and Isaacs, I. Martin and Jedwab, Andrea and Johnson, Kenneth and Karaali, Gizem and Lauve, Aaron and Le, Tung and Lewis, Stephen and Li, Huilan and Magaard, Kay and Marberg, Eric and Novelli, Jean-Christophe and Pang, Amy and Saliola, Franco and Tevlin, Lenny and Thibon, Jean-Yves and Thiem, Nathaniel and Venkateswaran, Vidya and Vinroot, C. Ryan and Yan, Ning and Zabrocki, Mike},
 Title = {Supercharacters, symmetric functions in noncommuting variables, and related {Hopf} algebras},
 FJournal = {Advances in Mathematics},
 Journal = {Adv. Math.},
 Volume = {229},
 Number = {4},
 Pages = {2310--2337},
 Year = {2012},
}

@Article{GebSag01,
 Author = {Gebhard, David and Sagan, Bruce},
 Title = {A chromatic symmetric function in noncommuting variables},
 FJournal = {Journal of Algebraic Combinatorics},
 Journal = {J. Algebr. Comb.},
 Volume = {13},
 Number = {3},
 Pages = {227--255},
 Year = {2001},
}

@InCollection{LodayRonco10,
 Author = {Loday, Jean-Louis and Ronco, Mar{\'{\i}}a},
 Title = {Combinatorial {Hopf} algebras.},
 BookTitle = {Quanta of maths. Conference on non commutative geometry in honor of Alain Connes, Paris, France, March 29--April 6, 2007},
 ISBN = {978-0-8218-5203-3},
 Pages = {347--383},
 Volume = {11},
 Year = {2010},
 Publisher = {Amer. Math. Soc., Providence, RI},
 Language = {English}
}

@book{L08,
  title={Generalized bialgebras and triples of operads},
  author={Loday, Jean-Louis},
  series={Ast\'erisque},
  pages={126},
  volume={320},
  year={2008},
  publisher={Numdam}
}

@Inbook{Cartier2007,
author="Cartier, Pierre",
editor="Cartier, Pierre
and Moussa, Pierre
and Julia, Bernard
and Vanhove, Pierre",
title="A Primer of Hopf Algebras",
bookTitle="Frontiers in Number Theory, Physics, and Geometry II: On Conformal Field Theories, Discrete Groups and Renormalization",
year="2007",
publisher="Springer Berlin Heidelberg",
address="Berlin, Heidelberg",
pages="537--615",
isbn="978-3-540-30308-4",
doi="10.1007/978-3-540-30308-4_12",
url="https://doi.org/10.1007/978-3-540-30308-4_12"
}

@article{Cartier57,
     author = {Cartier, P.},
     title = {Hyperalg\`ebres et groupes formels},
     journal = {S\'eminaire ``Sophus Lie''},
     note = {talk:2},
     pages = {1--6},
     publisher = {Secr\'etariat math\'ematique},
     volume = {2},
     year = {1955-1956},
     language = {fr},
     url = {https://www.numdam.org/item/SSL_1955-1956__2__A4_0/}
}

@article{P94,
  title={L'alg{\`e}bre des descentes d'une big{\`e}bre gradu{\'e}e},
  author={Patras, Fr{\'e}d{\'e}ric},
  journal={Journal of Algebra},
  volume={170},
  pages={547--566},
  year={1994},
  publisher={Elsevier}
}

@article{P99,
  title={Higher {L}ie Idempotents},
  author={Patras, Fr{\'e}d{\'e}ric and Reutenauer, Christophe},
  journal={Journal of Algebra},
  volume={222},
  pages={51--64},
  year={1999},
  issn = {0021-8693},
  doi = {https://doi.org/10.1006/jabr.1999.7887},
  publisher={Elsevier}
}

@article{PR02,
  title={Lie Representations and an Algebra Containing {S}olomon's},
  author={Patras, Fr{\'e}d{\'e}ric and Reutenauer, Christophe},
  journal={Journal of Algebraic Combinatorics},
  volume={16},
  issue={3},
  pages={301--314},
  year={2002},
  issn = {1572-9192},
  doi = {https://doi.org/10.1023/A:1021856522624},
  publisher={Elsevier}
}

@article{AL13,
  title={Lagrange's theorem for {H}opf monoids in species},
  author={Aguiar, Marcelo and Lauve, Aaron},
  journal={Canadian Journal of Mathematics},
  volume={65},
  number={2},
  pages={241--265},
  year={2013},
  publisher={Cambridge University Press}
}

\end{document}